\documentclass[12pt, a4paper]{amsart}

\usepackage{amsmath}
\usepackage{amssymb}
\usepackage{amscd} 
\usepackage{hyperref} 
\usepackage{verbatim} 
\usepackage{color}  
\usepackage{tikz-cd}\usetikzlibrary{babel} 
\usepackage{enumerate} 
\usepackage{cite}

\usepackage[english]{babel}
\title[Profinite invariants of arithmetic groups]
 {Profinite invariants of arithmetic groups}

 \author[H. Kammeyer]{Holger Kammeyer}
 \email{holger.kammeyer@kit.edu}
\address{Institute for Algebra and Geometry, Karlsruhe Institute of Technology, 76131 Karlsruhe, Germany}

 \author[S. Kionke]{Steffen Kionke}
 \email{steffen.kionke@kit.edu}
 \address{Institute for Algebra and Geometry, Karlsruhe Institute of Technology, 76131 Karlsruhe, Germany}

 \author[J. Raimbault]{Jean Raimbault}
  \email{Jean.Raimbault@math.univ-toulouse.fr}
\address{Institut de Math\'ematiques de Toulouse ; UMR5219 \\ Universit\'e de Toulouse ; CNRS \\ UPS IMT, F-31062 Toulouse Cedex 9, France}

 \author[R. Sauer]{Roman Sauer}
 \email{roman.sauer@kit.edu}
\address{Institute for Algebra and Geometry, Karlsruhe Institute of Technology, 76131 Karlsruhe, Germany}


\theoremstyle{plain}
\newtheorem{theorem}{Theorem}
\newtheorem{lemma}[theorem]{Lemma}
\newtheorem{corollary}[theorem]{Corollary}
\newtheorem{proposition}[theorem]{Proposition}

\theoremstyle{definition}
\newtheorem{definition}[theorem]{Definition}
\newtheorem{remark}[theorem]{Remark}
\newtheorem{example}[theorem]{Example}

\numberwithin{equation}{section}
\numberwithin{theorem}{section}

\DeclareMathOperator{\inn}{int}
\DeclareMathOperator{\Lie}{Lie}

\DeclareMathOperator{\Aut}{Aut}

\DeclareMathOperator{\Res}{Res}
\DeclareMathOperator{\Ad}{Ad}
\DeclareMathOperator{\sign}{sign}
\DeclareMathOperator{\rank}{rk}
\DeclareMathOperator{\pr}{pr}

\providecommand{\normal}{\trianglelefteq}
\providecommand{\calL}{\mathcal{L}}
\providecommand{\calD}{\mathcal{D}}
\providecommand{\calO}{\mathcal{O}}
\providecommand{\fg}{\mathfrak{g}}
\providecommand{\fh}{\mathfrak{h}}
\providecommand{\fp}{\mathfrak{p}}
\providecommand{\fk}{\mathfrak{k}}
\providecommand{\ft}{\mathfrak{t}}
\providecommand{\fa}{\mathfrak{a}}
\providecommand{\fb}{\mathfrak{b}}
\providecommand{\fc}{\mathfrak{c}}

\providecommand{\fso}{\mathfrak{so}}

\providecommand{\bbN}{\mathbb{N}}
\providecommand{\bbR}{\mathbb{R}}
\providecommand{\bbQ}{\mathbb{Q}}
\providecommand{\bbZ}{\mathbb{Z}}
\providecommand{\bbF}{\mathbb{F}}
\providecommand{\bbA}{\mathbb{A}}
\providecommand{\bbC}{\mathbb{C}}
\providecommand{\alg}[1]{\underline{#1}}
\DeclareMathOperator{\SL}{SL}

\DeclareMathOperator{\SO}{SO}
\DeclareMathOperator{\Gal}{Gal}
\DeclareMathOperator{\Spin}{Spin}
\DeclareMathOperator{\Spn}{\alg{\Spin}}
\DeclareMathOperator{\ann}{ann}
\DeclareMathOperator{\vol}{vol}

\providecommand{\ignore}[1]{}

\newcounter{commentcounter}

\usepackage{ifthen,srcltx}
\newcommand{\showcomments}{yes}

\newsavebox{\commentbox}
\newenvironment{comnz}%
{\ifthenelse{\equal{\showcomments}{yes}}%
{\footnotemark
        \begin{lrbox}{\commentbox}
        \begin{minipage}[t]{1.25in}\raggedright\sffamily\tiny
        \footnotemark[\arabic{footnote}]}
{\begin{lrbox}{\commentbox}}}
{\ifthenelse{\equal{\showcomments}{yes}}
{\end{minipage}\end{lrbox}\marginpar{\usebox{\commentbox}}}
{\end{lrbox}}}

\begin{document}
\selectlanguage{english}

\begin{abstract}
  We prove that the sign of the Euler characteristic of arithmetic groups with CSP is determined by the profinite completion.  In contrast, we construct examples showing that this is not true for the Euler characteristic itself and that the sign of the Euler characteristic is not profinite among general residually finite groups of type \(F\).  Our methods imply similar results for \(\ell^2\)-torsion as well as a strong profiniteness statement for Novikov--Shubin invariants. 
\end{abstract}

\maketitle

\section{Introduction}
A finitely generated, residually finite group \(\Gamma\) is called \emph{profinitely rigid} if
any other such group $\Lambda$ with the same set of finite quotients as $\Gamma$ is isomorphic
to $\Gamma$;  this can be expressed in terms of profinite completions: if
\(\widehat{\Lambda} \cong \widehat{\Gamma}\), then $\Lambda \cong \Gamma$ (see
\cite{DFPR}). While all finitely generated abelian groups have this property, there are already
virtually cyclic groups which are not profinitely rigid \cite{Baumslag}.  In general,
profinite rigidity is extremely difficult to characterize.  Recent work of
Bridson--McReynolds--Reid--Spitler \cite{BmRRS} shows that profinite rigidity holds for
certain Kleinian groups, including the Weeks manifold group. On the other hand
we note that profinite rigidity of free groups, surface groups or
$\SL_n(\mathbb Z)$ is still open.
\subsection{Main results}
Two related questions seem more accessible: (i) to establish profinite rigidity among a certain class of groups and (ii) to find {\em profinite invariants}. A group invariant is \emph{profinite} if it takes the same value on profinitely isomorphic groups. In this paper we study a combination of both: we establish profinite invariance of the sign of the Euler characteristic in the class of arithmetic groups with the congruence subgroup property. In particular,
this (conjecturally) includes all irreducible lattices in higher-rank semisimple Lie groups.

\begin{theorem} \label{thm:main-theorem}
  Let \(\alg{G}_1\) and \(\alg{G}_2\) be linear algebraic groups defined over number fields \(k_1\) and \(k_2\), and let \(\Gamma_1 \le \alg{G}_1(k_1)\) and \(\Gamma_2 \le \alg{G}_2(k_2)\) be arithmetic subgroups.  Suppose that \(\alg{G}_1\) and \(\alg{G_2}\) have a finite congruence kernel and that \(\Gamma_1\) is profinitely commensurable with \(\Gamma_2\).  Then \(\sign \chi(\Gamma_1) = \sign \chi(\Gamma_2)\). 
\end{theorem}

Let us explain the meaning of the terms in this statement. Two groups \(\Gamma_1\) and \(\Gamma_2\) are called \emph{profinitely commensurable} if the profinite completions \(\widehat{\Gamma_1}\) and \(\widehat{\Gamma_2}\) have isomorphic open subgroups. Equivalently, \(\Gamma_1\) and \(\Gamma_2\) have finite index subgroups with isomorphic profinite completions.  The function \(\sign(x)\) takes the values \(-1, 0, 1\) if \(x < 0,\ x = 0,\ x > 0\). A subgroup of $\alg G(k)$ is \emph{arithmetic} if it is commensurable to $\alg{G}(k) \cap \mathrm{GL}_n(\calO_k)$ for some $k$-rational embedding $\alg G \to \mathrm{GL}_n$, where $\calO_k$ is the ring of algebraic integers in $k$. It is more technical to define the congruence kernel of $\alg G$ and we will not do so here. Serre conjectured
that for simple groups $\alg{G}$ the congruence kernel is finite whenever
the real Lie group $\alg G(k \otimes_{\mathbb Q} \mathbb R)$ has real rank at least 2; we refer to \cite{Raghunathan_CSP} for a survey on the status of this conjecture. 

We note that it is known that profinite rigidity does not hold for all higher rank arithmetic lattices, even among themselves (as follows from \cite{Aka2012a}). However, the profinite isomorphism class of arithmetic groups for which the congruence subgroup property holds, is easier to understand than that of general lattices; for example M.~Aka proves in \emph{loc.\ cit.}\ that it is always finite within the class of arithmetic groups. To prove the theorem above we push Aka's arguments further.  

It is possible to calculate the Euler characteristic of arithmetic groups using Harder's Gau{\ss}-Bonnet Theorem \cite{Harder71}. We apply this method to obtain the following example which shows that Theorem \ref{thm:main-theorem} does not extend to the Euler characteristic itself.

\begin{theorem} \label{thm:euler-of-spin-groups}
  For positive integers \(m\) and \(n\), let \(\Gamma_{m,n}\) be the level four principal congruence subgroup of \(\Spin(m,n)(\bbZ)\).  Then \(\widehat{\Gamma_{8,2}} \cong \widehat{\Gamma_{4,6}}\) but
    \[ \chi(\Gamma_{8,2}) = 2^{89} \cdot 5^2 \cdot 17 \quad \text{whereas} \quad \chi(\Gamma_{4,6}) = 2^{90} \cdot 5^2 \cdot 17. \]
\end{theorem}

The spinor groups \(\Spin(m,n)(\bbZ)\) arise from the \((m+n)\)-ary integral diagonal quadratic form with \(m\) coefficients ``\(+1\)'' and \(n\) coefficients ``\(-1\)''.  Precise definitions are given in Section~\ref{section:euler-computation}.  The existence of the above examples implies that one cannot broaden the conclusion of Theorem~\ref{thm:main-theorem} from arithmetic to residually finite groups that admit a finite classifying space. The latter is referred to as being of type $(F)$. 

\begin{corollary} \label{cor:example-signs}
  There are three residually finite groups $\Gamma_1$, $\Gamma_2$, and $\Gamma_3$
  of type (F) which have isomorphic profinite completions such that
  $$ \chi(\Gamma_1) < 0 ,\qquad \chi(\Gamma_2) = 0 ,\qquad \chi(\Gamma_3) > 0.$$
\end{corollary}

Setting \(c = 2^{89} \cdot 5^2 \cdot 17\), the above groups can simply be taken as
  \begin{align*}
  \Gamma_1 &= (\Gamma_{8,2}\times\Gamma_{8,2})* F_{2c^2},\\
  \Gamma_2 &= (\Gamma_{8,2}\times\Gamma_{4,6})* F_{2c^2},\\
  \Gamma_3 &= (\Gamma_{4,6}\times\Gamma_{4,6})* F_{2c^2},
  \end{align*}
  where \(F_{2c^2}\) is the free group on \(2c^2\) letters.  Since the profinite completion functor preserves products and coproducts, the three groups are profinitely isomorphic. They are still residually finite and of type~\((F)\). Additivity and multiplicativity of the Euler characteristic gives
    \begin{align*}
    \chi(\Gamma_1) &= c^2 + (1-2c^2) - 1 = -c^2 < 0,\\
    \chi(\Gamma_2) &= 2c^2 + (1-2c^2) -1 = 0, \\
    \chi(\Gamma_3) &= 4c^2 + (1-2c^2) - 1 = 2c^2 > 0.
  \end{align*}
 
 The Euler characteristic equals the alternating sum of the \emph{\(\ell^2\)-Betti numbers}~\cite{Lueck:l2-invariants, Kammeyer:intro-l2}.  For arithmetic groups, \(\ell^2\)-Betti numbers are known to be nonzero in at most one degree.  Such a nonzero \(\ell^2\)-Betti number occurs if and only if the group is semisimple and the fundamental rank is zero.  In that case the degree with nonvanishing \(\ell^2\)-Betti number is given by half the dimension of the associated symmetric space.  This dimension, however, can change when passing to a profinitely commensurable arithmetic group.  So \(\ell^2\)-Betti numbers themselves are not profinite.  Among \(S\)-arithmetic groups, no higher \(\ell^2\)-Betti number is profinite~\cite{Kammeyer-Sauer:spinor-groups}, in contrast to the first \(\ell^2\)-Betti number which is profinite among all finitely presented residually finite groups~\cite[Corollary~3.3]{Bridson-Conder-Reid:Fuchsian}.  Thus in the semisimple case, the proof of Theorem~\ref{thm:main-theorem} splits into two parts: showing that the fundamental rank is profinite, so that the vanishing of Euler characteristic is profinite, and showing that the profinite completion determines the dimension of the symmetric space mod \(4\).  

\subsection{Extension to other invariants} 
Whenever an arithmetic group~\(\Gamma\) has vanishing Euler characteristic, a secondary invariant called \emph{\(\ell^2\)-torsion} and denoted by \(\rho^{(2)}(\Gamma)\) is defined; see~\cite[Chapter~3]{Lueck:l2-invariants} and~\cite[Chapter~5]{Kammeyer:intro-l2} for an introduction.  In many ways, \(\rho^{(2)}(\Gamma)\) behaves like an ``odd-dimensional cousin'' of \(\chi(\Gamma)\).  Also the profinite behavior of \(\rho^{(2)}(\Gamma)\) is parallel to \(\chi(\Gamma)\).

\begin{theorem} \label{thm:sign-of-l2-torsion}
  In addition to the assumptions in Theorem~\ref{thm:main-theorem}, suppose that \(\chi(\Gamma_i)= 0\) for either (then both) \(i = 1,2\) and \(\rank_{k_i} \alg{G}_i = 0\) for both \(i = 1,2\).  Then \(\sign \rho^{(2)}(\Gamma_1) = \sign \rho^{(2)} (\Gamma_2)\).
\end{theorem}

We conjecture that the assumption on \(\rank_{k_i} \alg{G}_i\) is not needed.  It would not be needed if~\cite[Conjecture~1.2]{Lueck-Sauer-Wegner:determinant} was true and it is not needed if the \emph{fundamental rank} of \(\alg{G}_i\) defined on p.\,\pageref{page:fundamental-rank} is even~\cite[Theorem~1.2]{Kammeyer:nonuniform}.  But in our proof, we are using the equality of analytic and cellular \(\ell^2\)-torsion which is, at present, only known if \(\Gamma_i\) is a cocompact lattice in the Lie group \(\prod_v \alg{G}_i({k_i}_v)\) where \(v\) runs through the infinite places of \(k_i\).  This cocompactness condition is equivalent to \(\rank_{k_i} \alg{G}_i = 0\).

If \(\Gamma\) and \(\Lambda\) are of type~\((F)\) and \(\Lambda\) is residually finite and \mbox{\(\ell^2\)-acyclic}, then we have the product formula \(\rho^{(2)}(\Gamma \times \Lambda) = \chi(\Gamma) \rho^{(2)}(\Lambda)\) as proven in~\cite[Theorem~3.93\,(4), p.\,161]{Lueck:l2-invariants}.  Hence if \(M\) is some closed hyperbolic \(3\)-manifold, then
\[ 2 \,\rho^{(2)}(\Gamma_{8,2} \times \pi_1 M) = \rho^{(2)}(\Gamma_{4,6} \times \pi_1 M) < 0. \]
The groups \(\Lambda_i = \pi_1 M \times \Gamma_{4-i}\) are residually finite, of type \((F)\) and
\[ \rho^{(2)}(\Lambda_1) < 0 ,\qquad \rho^{(2)}(\Lambda_2) = 0 ,\qquad \rho^{(2)}(\Lambda_3) > 0.\]
This shows that, as before, Theorem~\ref{thm:sign-of-l2-torsion} has no immediate extension in one way or another.

Since an arithmetic group \(\Gamma\) has at most one nonzero \(\ell^2\)-Betti number, the Euler characteristic \(\chi(\Gamma)\) encodes the entire \emph{reduced} \(\ell^2\)-co\-homo\-logy.  The lesser known \emph{Novikov--Shubin invariants} \(\alpha_p(\Gamma)\) capture whether \(\Gamma\) additionally possesses \emph{unreduced} \(\ell^2\)-cohomology.  The reader can find an overview in~\cite[Chapter~2]{Lueck:l2-invariants}.  In the semisimple and \(k\)-anisotropic case, our methods imply an even stronger statement on these subtle invariants.  To state it, let us introduce the relabeling \(\overline{\alpha}_{\pm q} (\Gamma) = \alpha_{k \pm q}(\Gamma)\) where the symmetric space on which \(\Gamma\) acts is either \(2k\)- or \((2k + 1)\)-dimensional.

\begin{theorem} \label{thm:novikov-shubin}
  For \(i = 1, 2\), let \(k_i\) be number fields, let \(\alg{G}_i\) be semisimple linear algebraic \(k_i\)-groups with \(\rank_{k_i} \alg{G}_i = 0\), and let \(\Gamma_i \le \alg{G}_i\) be arithmetic.  Suppose that \(\alg{G}_1\) and \(\alg{G}_2\) have a finite congruence kernel and that \(\Gamma_1\) is profinitely commensurable with \(\Gamma_2\).  Then \(\overline{\alpha}_{\pm q} (\Gamma_1) = \overline{\alpha}_{\pm q} (\Gamma_2)\) for all \(q\).
\end{theorem}

This time the assumption that \(\rank_{k_i} \alg{G}_i = 0\) is likely to be essential because only in the cocompact case do analytic and cellular Novikov--Shubin invariants agree~\cite{Efremov:cell-decompositions} and only the analytic Novikov-Shubin invariants are entirely governed by the fundamental rank. Compare~\cite[Theorem~1.4]{Kammeyer:thesis}.  Given a semisimple Lie group \(G\) with symmetric space \(X = G/K\), let us set \(n = \dim X\) and let \(m = \delta(G)\) be the fundamental rank.  For a torsion-free cocompact lattice \(\Gamma \le G\), Olbrich~\cite[Theorem~1.1.(b)]{Olbrich:l2-symmetric} has shown in the analytic approach that \(\alpha_p(\Gamma) \neq \infty^+\) if and only if \(p \in [\frac{n-m}{2}, \frac{n+m}{2} -1]\).  Moreover, in this range we have \(\alpha_p(\Gamma) = m\).  By Selberg's lemma, the arithmetic groups \(\Gamma_i\) have torsion-free subgroups of finite index. Novikov--Shubin invariants are unchanged when passing to commensurable groups~\cite[Theorem~2.55\,(6)]{Lueck:l2-invariants}.  Since we show in Theorem~\ref{thm:f-rank-dim-mod-4} that \(m = \delta(G)\) is a profinite invariant for arithmetic subgroups of semisimple groups, Theorem~\ref{thm:novikov-shubin} follows. 

\subsection{Towards $S$-arithmetic groups and weakening CSP}
In general we do not know whether our results generalize from arithmetic to \(S\)-arithmetic groups. However,
we can extend our results in special cases. For example, for groups
\[ \Gamma_i = \Spin(q_i)(\bbZ[S_i^{-1}]) \]
where \(S_i\) are finite sets of rational primes and \(q_i\) are integral quadratic forms such that \(\Spin(q_i)\) has finite \(S_i\)-congruence kernel, we checked that still \(\sign \chi(\Gamma_1) = \sign \chi(\Gamma_2)\) whenever \(\Gamma_1\) and \(\Gamma_2\) are profinitely commensurable.  The proof is a case by case study invoking the classification of anisotropic quadratic forms over \(\bbQ_p\).  Interestingly and as opposed to the arithmetic case, for these \(S\)-arithmetic groups it is no longer true that the dimension of the symmetric space is a profinite invariant mod~4.  However, if \(\Gamma_1\) and \(\Gamma_2\) are profinitely commensurable and \(\dim X_1 \not\equiv \dim X_2 \bmod 4\), then there always exists a finite prime \(p \in S_1 \cap S_2\) such that \(\rank_{\bbQ_p} \alg{G}_1 \not\equiv \rank_{\bbQ_p} \alg{G}_2 \bmod 2\) so that still \(\sign \chi (\Gamma_1) = \sign \chi (\Gamma_2)\).  An example of this behavior is presented in Example \ref{dim_infty_not_profin}.

Another family of $S$-arithmetic groups for which we can establish profiniteness of the sign of the Euler characteristic is the following: fixing a (higher rank simple) \(\bbQ\)-group \(\alg{G}\), non-commensurable but profinitely commensurable \(S\)-arithmetic groups occur when \(\alg{G}\) is considered over varying number fields.  Methods due to Aka~\cite{Aka2012a} are used in~\cite{Kammeyer:commensurable} to show that these groups must be defined over arithmetically equivalent number fields \(k\) and \(l\). This implies profiniteness of \(\sign \chi(\Gamma)\) if \(S\) contains no places over ramified primes or if \(\alg{G}\) splits over \(k\) and \(l\).

It is unclear if these observations can be extended to general algebraic groups with CSP.  Notwithstanding, we can strengthen Theorems~\ref{thm:main-theorem} and~\ref{thm:sign-of-l2-torsion} formally by only requiring that one of the two groups be arithmetic and have CSP.

\begin{theorem} \label{thm:strengthening}
  Let \(\alg{G}_1\) and \(\alg{G}_2\) be linear algebraic groups defined over number fields \(k_1\) and \(k_2\).  Suppose \(\alg{G}_1\) has finite congruence kernel and that either $\alg G_2$ is reductive and each $k_2$-simple factor of the universal covering of its derived subgroup satisfies the Platonov--Margulis conjecture, or that $\alg G_2$ is not reductive.  Let \(\Gamma_1 \le \alg{G}_1\) be arithmetic and let \(\Gamma_2 \le \alg{G}_2\) be \(S\)-arithmetic for a finite set of places \(S\) of \(k_2\) containing all the infinite ones.
  \begin{enumerate}[(i)]
  \item If \(\Gamma_1\) and \(\Gamma_2\) are profinitely commensurable, then
    \[ \sign \chi(\Gamma_1) = \sign \chi(\Gamma_2). \]
  \item If in addition \(\rank_{k_1} \alg{G}_1 = \rank_{k_2} \alg{G}_2 = 0\) and \(\chi(\Gamma_1) = \chi(\Gamma_2) = 0\), then
    \[ \sign \rho^{(2)}(\Gamma_1) = \sign \rho^{(2)}(\Gamma_2). \]
    \end{enumerate}
\end{theorem}

See \cite{PlatonovRapinchuk} for an introduction to the Platonov--Margulis conjecture and \cite[Appendix A]{Rapinchuk_Segev} for a shorter and more up-to-date survey. We note that while this conjecture is still open in some case its status is still better than that of the congruence subgroup property; in particular it is known to hold for inner forms of type $A_n$. Unlike Theorem~\ref{thm:main-theorem} the above result can be applied when the $\mathbb R$-points of the Weil restriction of  $\alg G_2$ are of real rank one; Theorem~\ref{thm:main-theorem} is not applicable since real and complex hyperbolic lattices often do not have CSP (and are conjectured to never have it). 

\subsection{Comments on rank one groups}
Finally, the question occurs whether the assumption of CSP in Theorem \ref{thm:main-theorem} can be removed, which would essentially boil down to understanding the case of rank 1 semisimple Lie groups. Taking the classification of rank 1 semisimple real Lie groups into account, the profiniteness of the sign of Euler characteristic or \( L^2 \)-torsion reduces to the question of profiniteness of the dimension of the symmetric space modulo~4. 

However, the techniques used to prove such a statement would by necessity be very different from the rigidity results used in higher rank, except possibly for lattices in the quaternionic hyperbolic spaces and the octonionic hyperbolic plane. There has already been much work on this topic, or related topics; some topological profinite invariants for 3--manifold groups are given in \cite{Boileau_Friedl}. Let us also mention the following results of interest. 

\begin{enumerate}
\item Recent work of Bridson--McReynolds--Reid--Spitler \cite{BmRRS} shows that profinite rigidity holds for certain Kleinian groups, including the Weeks manifold group. Note that profinite rigidity is generally hard to establish. It is, for example, open whether free groups or more generally Fuchsian groups, Kleinian groups, \(\textup{SL}(n,\bbZ)\), or mapping class groups of closed surfaces are profinitely rigid. 
\item The question becomes more accessible if one only asks for profinite rigidity \emph{among} a certain class of groups. In this vein, Bridson--Reid~\cite{Bridson-Reid:figure-eight} had previously shown that the figure-eight knot group is a Kleinian group which is profinitely rigid among 3-manifold groups. In general, it is not even known whether Kleinian groups are profinitely rigid among themselves.  In fact, it is open whether the volume of hyperbolic 3-manifolds is \emph{profinite} in the sense that it agrees for two such manifolds whose fundamental groups have isomorphic profinite completions.
\item Fuchsian groups are profinitely rigid among lattices in Lie groups and S-arithmetic groups; this follows from profiniteness invariance of the first \( L^2 \)-Betti number, which distinguishes them from lattices in other Lie groups, and the work of Bridson--Conder--Reid \cite{Bridson-Conder-Reid:Fuchsian} which distinguishes them between themselves. 

\item It follows from the work of Bergeron--Haglund--Wise \cite{Bergeron_Haglund_Wise} and Minasyan--Zaleskii \cite{Minasyan_Zalesskii} that arithmetic lattices of simple type in $\mathrm{SO}(n, 1)$ (for any $n$) are cohomologically good, in particular their profinite completion knows their virtual cohomological dimension, which equals $n$ for a uniform lattice and $n - 1$ for a non-uniform one.
  It is well-known that Fuchsian groups are good. It also follows from the work of Agol \cite{Agol_VH} together with that of Minasyan--Zalesskii that lattices in \( \mathrm{SO}(3, 1) \) are cohomologically good.

\item Recently M.~Stover \cite{Stover} gave examples, for any $n \ge 2$, of a pair of lattices in $\mathrm{PU}(n, 1)$ which are profinitely isomorphic, but not commensurable to each other. 
\end{enumerate}
\subsection{Outline}
In Section~\ref{section:profinite-invariance} we establish profinite invariance of the fundamental rank of the associated Lie groups as well as profinite invariance of the dimension of the associated symmetric space mod~4 in the semisimple case.  Section~\ref{section:conclusion} then derives the main results as straightforward conclusions from the previous section.  In Section~\ref{section:euler-computation} we explicitly compute the Euler characteristic of the arithmetic spin groups \(\Gamma_{m,n}\).

\subsection*{Acknowledgements}
H.K., S.K., and R.S.~acknowledge 
funding by the Deutsche Forschungsgemeinschaft (DFG, German Research Foundation) –~281869850 (RTG 2229) and 338540207. 
 J.R.~acknowledges support by ANR grant ANR-16-CE40-0022-01 - AGIRA.  


\section{Profinite invariance of fundamental rank and dimension mod 4} \label{section:profinite-invariance}
For better reference, we have chosen to formulate our results in the introduction in terms of number fields.  But as is well known, given a number field \(k\) and a linear algebraic \(k\)-group \(\alg{G}\), the restriction of scalars functor \(\Res^k_{\bbQ}\) as for instance introduced in~\cite[Section~2.1.2, p.\,49]{PlatonovRapinchuk} comes with a natural isomorphism \(\Res^k_{\bbQ} \alg{G} (\bbQ) \cong \alg{G}(k)\) which preserves the notion of arithmetic subgroup and satisfies, moreover, \(C(\Res^k_{\bbQ} \alg{G}, \,\bbQ) \cong C(\alg{G},\,k)\) for the congruence kernels.
Thus every arithmetic subgroup of a \(k\)-group is isomorphic to an arithmetic subgroup of a \(\bbQ\)-group and the former has finite congruence kernel if and only if the latter does.  These remarks justify that henceforth we will work over \(k = \bbQ\) only.  As an outcome of the introduction we see that the following theorem is the main technical result we need to attack.

\begin{theorem}\label{thm:f-rank-dim-mod-4}
  Let \(\Gamma_1 \le \alg{G}_1\) and \(\Gamma_2 \le \alg{G}_2\) be arithmetic subgroups of semisimple linear algebraic \(\bbQ\)-groups with finite congruence kernel.  If \(\Gamma_1\) is profinitely commensurable with \(\Gamma_2\), then
  \begin{enumerate}[(i)]
  \item $\dim X_1 \equiv \dim X_2 \bmod 4$ and
  \item $\delta(G_1) = \delta(G_2)$.
  \end{enumerate}
\end{theorem}

Here \(X_i = G_i/K_i\) is the \emph{symmetric space} associated with \(\alg{G}_i\).  It is defined by choosing a maximal compact subgroup \(K_i \subseteq G_i\) of the Lie group \(G_i = \alg{G}_i(\bbR)\).  The number
\[ \delta(G_i) =  \rank_\bbC (\mathcal{L}(G_i) \otimes_\bbR \bbC) - \rank_\bbC (\mathcal{L}(K_i) \otimes_\bbR \bbC) \]
is called the \emph{fundamental rank} of \(X_i\), sometimes also known as the \emph{deficiency} of \(G_i\).  The notation \(\mathcal{L}(G_i)\) and \(\mathcal{L}(K_i)\) denotes the Lie algebras of the Lie groups \(G_i\) and \(K_i\).\label{page:fundamental-rank}

The rough outline of the proof of Theorem~\ref{thm:f-rank-dim-mod-4} is as follows.  We first show that under the assumption of CSP and strong approximation, the profinite commensurability of \(\Gamma_1\) and \(\Gamma_2\) implies that the Lie algebras of the \(p\)-adic analytic groups \(\alg{G}_1(\bbQ_p)\) and \(\alg{G}_2(\bbQ_p)\) are isomorphic for all finite primes~\(p\) (Proposition~\ref{prop:Lie-algebras-isomorphic}).  Weil's product formula expresses the signature of a rational quadratic form mod 8 in terms of Gaussian sums associated with the \(\bbF_p\)-reductions of the form.  Applying this formula to the Killing forms of the Lie algebras of \(\alg{G}_1\) and \(\alg{G}_2\), we can conclude that \(\dim X_1 \equiv \dim X_2 \bmod 4\) (Proposition~\ref{prop:same-sign}).  To show that \(\delta(G_1) = \delta(G_2)\), we first explain that if \(\alg{G}_1(\bbQ_p) \cong \alg{G}_2(\bbQ_p)\) for all \(p\), then of necessity \(\alg{G}_1 \times_\bbQ \bbR\) and \(\alg{G}_2 \times_\bbQ \bbR\) are inner forms of one another (Proposition~\ref{prop:inner-forms}).  Therefore, if we fix an isomorphism \(\varphi \colon \alg{G}_1 \times_\bbQ \bbC \rightarrow \alg{G}_2 \times_\bbQ \bbC\) and if \(\tau_i\) denotes the complex conjugation satisfying \(\alg{G}_i(\bbC)^{\tau_i} = \alg{G}_i(\bbR)\), then \(\tau_1\) and \(\varphi^{-1} \tau_2 \varphi\) are conjugate by an inner automorphism.  In that case the maximal compact subgroups of \(\alg{G}_i(\bbR)\) have the the same rank as we verify in Proposition~\ref{prop:same-rank}, so that equality of fundamental ranks follows.

To begin with, we verify that in a typical situation, isomorphisms of products of \(p\)-adic Lie groups must be factor-wise.  The notation ``\(\le_o\)'' and ``\(\le_c\)'' indicates open and closed subgroups, respectively.

\begin{lemma}\label{lem:isomorphic-lie}
  Let $G$ be a profinite group with open subgroups $A,B \leq_o G$.
  Suppose that 
  $$ A = \prod_{p} G_p  \quad \text{ and } \quad B = \prod_{p} H_p, $$
  for certain $p$-adic analytic groups $G_p, H_p \leq_c G$ where the product runs over all
  prime numbers.  Then $H_p$ and $G_p$ are virtually isomorphic for all primes $p$.
  In particular, the Lie algebras $\calL(H_p)$ and $\calL(G_p)$ are $\bbQ_p$-isomorphic.
\end{lemma}
\begin{proof}
  Let $p$ be a prime. We may assume that $\dim(G_p) \geq \dim(H_p)$. After possibly shrinking $B$, we may assume
  $B \subseteq A = G$.  Let $\pi \colon G \to G_p$ denote the projection homomorphism.  For
  every prime $\ell \neq p$ the image $\pi(H_\ell)$ is $p$-adic and $\ell$-adic analytic,
  hence a finite group.  Moreover, let $U_p \normal_o G_p$ be an open normal uniform pro-$p$
  subgroup. Recall that such a group is torsion-free.  The group $H_\ell$, however, does not
  contain pro-$p$ elements of infinite order and thus $\pi(H_\ell) \cap U_p = \{1\}$.

  Now we consider the homomorphism $\overline{\pi}\colon G \to G_p/U_p$ composed from $\pi$
  and the canonical factor map $G_p \to G_p/U_p$.  Since $\overline{\pi}$ is continuous, there
  is a finite set of primes $S$ (with $p \in S$) such that
  $$\overline{\pi}(\prod_{\ell \not\in S} H_\ell) = \{1\}.$$
  It follows that $\pi(\prod_{\ell\neq p} H_\ell)$ is finite.  However, the homomorphism $\pi$
  is surjective and we deduce that $\pi(H_p)$ is an open subgroup of $G_p$.

  Choose an open normal uniform pro-$p$ subgroup $V_p \normal_o H_p$ such that $\pi(V_p) \subseteq U_p$.
  Since $V_p$ is finitely generated powerful and $U_p$ is torsion-free, we deduce that $\pi(V_p)$
  is a finitely generated, powerful (see \cite[Definition 2.1]{DDMS2003}), torsion-free pro-$p$ group. By \cite[Theorem~4.5]{DDMS2003} we get that
  $\pi(V_p)$ is a uniform subgroup in $U_p$, and as it is also open we have $\dim \pi(V_p) = \dim U_p$, so $\dim V_p \ge  \dim U_p$. By assumption $\dim(V_p) \leq \dim(U_p)$ and we conclude that
  the dimensions are equal and that $\pi|_{V_p}$ is an isomorphism onto its image. 
\end{proof}

\begin{remark}\label{rem:algebraic-vs-analytic}
For an affine group scheme $\alg{H}$ over a commutative ring $R$, the Lie algebra (functor)
will be denoted by $\Lie(\alg{H})$.
For a Lie group $U$ over a complete valuated field $k$, e.g.  $\bbR$, $\bbC$ or $\bbQ_p$, the
associated $k$-Lie algebra will be denoted by $\calL(U)$.
Recall that, if $\alg{G}$ is a linear algebraic group over $k$,
then $\alg{G}(k)$ is a $k$-analytic Lie group and 
$$ \Lie(\alg{G})(k) \cong \calL(\alg{G}(k)).$$
\end{remark}


A little more general than necessary, we will now see that assuming CSP and strong approximation, profinitely commensurable \(S\)-arithmetic subgroups lie in algebraic groups whose Lie algebras become isomorphic when completing the field outside \(S\). 

\begin{proposition}\label{prop:Lie-algebras-isomorphic}
  Let \(S_1\) and \(S_2\) be finite sets of places of \(\bbQ\) containing the infinite one and let
  \(\alg{G}_1\) and \(\alg{G_2}\) be algebraic \(\bbQ\)-groups.
  Assume  $\alg{G}_i$ has finite $S_i$-congruence kernel and strong approximation
        w.\,r.\,t.~$S_i$.  
  Suppose \(\alg{G}_1\) and \(\alg{G}_2\) have profinitely commensurable \(S_1\)- and \(S_2\)-arithmetic subgroups. 
  Then \(S_1 = S_2\) and $\Lie(\alg{G}_1)(\bbQ_p) \cong \Lie(\alg{G}_2)(\bbQ_p)$ for $p \not\in S_1$.
\end{proposition}

\begin{proof}
Choose \(S_i\)-arithmetic subgroups $\Gamma_i \subseteq \alg{G}_i(\bbQ)$ such that $\widehat{\Gamma}_1 \cong \widehat{\Gamma}_2$. Since the congruence kernels are finite, we can pass to finite index subgroups if need be, to assume that $\widehat{\Gamma}_i$ is (isomorphic to) the closure of $\Gamma_i$ in $\alg{G}_i(\bbA_{S_i})$. Strong approximation implies that $\widehat{\Gamma}_i$ is an open subgroup of $\alg{G}_i(\bbA_{S_i})$ where \(\bbA_{S_i} = \prod_{p \not\in S_i} (\bbQ_p:\bbZ_p)\) denotes the ring of \(S_i\)-adeles. In particular, it has an open subgroup which is isomorphic to a product $\prod_{p\not\in S_i} U^{(i)}_p$ for certain open compact subgroups $U^{(i)}_p \leq_o \alg{G}_i(\bbQ_p)$.  So clearly \(S_1 = S_2\) and Lemma~\ref{lem:isomorphic-lie} and Remark~\ref{rem:algebraic-vs-analytic} complete the proof.
\end{proof}

The following is a variation of an observation of Rohlfs-Speh \cite[Lemma 2.5]{RohlfsSpeh1989}; see also
 \cite[Lemma 4]{Kionke2014}.
\begin{proposition}\label{prop:same-sign}
  Let $\alg{G}_1$ and $\alg{G}_2$ be semisimple linear algebraic groups over $\bbQ$ with associated
  symmetric spaces $X_1 = G_1/K_1$ and $X_2 = G_2/K_2$.  If $\Lie(G)(\bbQ_p) \cong \Lie(H)(\bbQ_p)$ for all $p$, then \[\dim X_1 \equiv \dim X_2 \bmod 4.\]
\end{proposition}
\begin{proof}
  We note that $\Lie(\alg{G}_i)(k) = k \otimes_\bbQ \Lie(\alg{G}_i)(\bbQ)$ for every extension field $k$ of
  $\bbQ$.  The Killing forms $\beta_i$ on $\Lie(\alg{G}_i)(\bbQ)$ are non-degenerate symmetric bilinear forms defined over $\bbQ$.  The Cartan decomposition
  implies that $\beta_i$ has signature
  $(\dim(X_i), \dim(K_i))$ as a form on $\Lie(\alg{G}_i)(\bbR)$.

  The Killing form is completely determined by the Lie algebra structure, hence the quadratic
  spaces $(\Lie(\alg{G}_1)(\bbQ_p),\beta_1)$ and $(\Lie(\alg{G}_2)(\bbQ_p),\beta_2)$ are isometric for every prime
  number $p$.  Weil's product formula implies that
  $\dim(X_1)-\dim(K_1) \equiv \dim(X_2)-\dim(K_2) \bmod 8$; see \cite[Corollary 8.2]{Scharlau1985}.  Let $d = \dim(G_1) = \dim (G_2)$, then
  $d = \dim(X_1) + \dim(K_1) = \dim(X_2) + \dim(K_2)$ and we deduce that
  \[2 \dim(X_1) \equiv \dim(X_1) - \dim(K_1) + d \equiv 2 \dim(X_2) \bmod 8. \qedhere \]
\end{proof}


\begin{definition}
  Let $\alg{G}_1$, $\alg{G}_2$ be linear algebraic groups over a field $k$ of characteristic $0$. We say that $\alg{G}_2$ is an \emph{inner form} of $\alg{G}_1$, if there is an isomorphism $\varphi \colon \alg{G}_1 \times_k \overline{k} \to \alg{G}_2 \times_k \overline{k}$ such that the associated $1$-cocycle $a$ defined as
  \[
  a_\sigma = \varphi^{-1} \sigma \varphi \sigma^{-1} \in \Aut(\alg{G}_1 \times_k \overline{k})
  \]
  for all $\sigma \in \Gal(\overline{k}/k)$ is equivalent to a cocycle with values in the group of inner automorphisms of $\alg{G}_1 \times_k \overline{k}$.
\end{definition}

\begin{proposition}\label{prop:inner-forms}
  Let $S$ be a finite set of places of $\bbQ$ containing the archimedean place. Let $\alg{G}_1$ and $\alg{G}_2$ be simply connected semisimple algebraic groups over $\bbQ$ such that \(\Lie(\alg{G}_1)(\bbQ_p) \cong \Lie(\alg{G}_2)(\bbQ_p)\) for all \(p \notin S\). Then $\alg{G}_2 \times_\bbQ \bbR$ is an inner form of $\alg{G}_1 \times_\bbQ \bbR$.
\end{proposition}

\begin{proof}
  The condition \(\Lie(\alg{G}_1)(\bbQ_p) \cong \Lie(\alg{G}_2)(\bbQ_p)\) for all primes $p \not\in S$ implies in particular that $\Lie(\alg{G}_1)(L) \cong \Lie(\alg{G}_2)(L)$ for some finite Galois extension $L/\bbQ$.  Since simply connected semisimple groups are determined up to isomorphism by their Lie algebras, we deduce
  \[
  \alg{G}_1 \times_\bbQ L \cong \alg{G}_2 \times_\bbQ L.
\]
Without loss of generality we may assume, after possibly passing to a larger field $L$, that
$ \alg{G}_1 \times_\bbQ L $ is split.
Note that for \emph{any} extension field $L'$ of $L$
there is a short exact sequence
\[ 1 \longrightarrow \Ad(\alg{G}_1\times_\bbQ L') \longrightarrow \Aut(\alg{G}_1 \times_\bbQ L')
  \stackrel{\pi_{L'}}{\longrightarrow} \Aut(\mathrm{Dyn}(\Phi)) \longrightarrow 1 \]
where $\mathrm{Dyn}(\Phi)$ denotes the Dynkin diagram of the root system $\Phi$ of $\alg{G}_1 \times_\bbQ L$; see
(25.16) in \cite{BookOfInvolutions}.
In addition, for every Galois extension $L_1/L_2$ where $L_1$ contains $L$,
the short exact sequence is Galois equivariant, where the
action of the Galois group $\Gal(L_1/L_2)$ on $\Aut(\mathrm{Dyn}(\Phi))$ is the one
induced from the action of $\Gal(L/\bbQ)$.

We choose an isomorphism
$\varphi\colon \alg{G}_1\times_\bbQ L \to \alg{G}_2\times_\bbQ L$ and consider the
corresponding $1$-cocycle defined by
$a_\sigma = \varphi^{-1} \sigma \varphi \sigma^{-1} \in \Aut(\alg{G}_1\times_\bbQ L)$ for all
$\sigma \in \Gal(L/\bbQ)$. The associated non-abelian cohomology class in
$H^1(\Gal(L/\bbQ),\Aut(\alg{G}_1\times_\bbQ L))$ will be denoted by $[a]$ and is independent
of the choice of \(\varphi\).  Pick an embedding $\iota\colon L \to \bbC$. We note that
$\iota$ induces a homomorphism $\iota^*\colon \Gal(\bbC/\bbR) \to \Gal(L/\bbQ)$ of groups.

  Suppose that $\alg{G}_2 \times_\bbQ \bbR$ is \emph{not} an inner form of $\alg{G}_1 \times_\bbQ \bbR$. In this case the image $\iota(L)$ is not contained in $\bbR$, since otherwise $\alg{G}_1 \times_\bbQ \bbR$ and $\alg{G}_2 \times_\bbQ \bbR$ are isomorphic. The long exact sequence
  \begin{align*}
    H^1\bigl(\Gal(\bbC/\bbR),\Ad(\alg{G}_1\times_\bbQ \bbC)\bigr) &\longrightarrow
    H^1\bigl(\Gal(\bbC/\bbR),\Aut(\alg{G}_1\times_\bbQ \bbC)\bigr) \\
    &\longrightarrow H^1\bigl(\Gal(\bbC/\bbR),\Aut(\mathrm{Dyn}(\Phi))\bigr)
  \end{align*}
  shows that the class ${\pi_\bbC}({\iota^*}^*[a]) \in H^1(\Gal(\bbC/\bbR), \Aut(\mathrm{Dyn}(\Phi)))$ is non-trivial. Let $\tau \in \Gal(\bbC/\bbR)$ denote complex conjugation.  By Chebotarev's density theorem, see \cite[Theorem~13.4]{Neukirch}, there is a prime number $p \not\in S$ and a prime ideal $\fp \subseteq \calO_L$ lying over $p$ such that $L_\fp / \bbQ_p$ is an unramified quadratic extension and the image of $\Gal(L_\fp / \bbQ_p) \to \Gal(L/\bbQ)$ is $\langle \iota^*\tau \rangle$.  Let $j \colon L \to L_\fp$ denote the inclusion map.  Naturality of the above long exact sequence shows that the cohomology class $\pi_{L_\fp}({j^*}^*([a])) \in H^1(\Gal(L_\fp/\bbQ_p), \Aut(\mathrm{Dyn}(\Phi)))$, and hence also ${j^*}^*([a]) \in H^1(\Gal(L_\fp/\bbQ_p), \Aut(\alg{G}_1 \times_\bbQ L_\fp))$ is non-trivial.  By~\cite[XXIV 7.3.1 (iii)]{SGA3-III}, the natural map
  \[
  d \colon \Aut_{L_\fp}(\alg{G}_1 \times_\bbQ L_\fp) \to \Aut_{L_\fp\text{-}\Lie}(\Lie(\alg{G}_1)(L_\fp))
  \]
  is an isomorphism.  It follows that the $\bbQ_p$-Lie algebras $\Lie(\alg{G}_1)(\bbQ_p)$ and $\Lie(\alg{G}_2(\bbQ_p))$ are not isomorphic which yields a contradiction.
\end{proof}


The next proposition will show that inner forms of real Lie groups have maximal compact subgroups of the same rank.  We will use in this context that inner forms can be realized as the fixed point sets of conjugate involutions on the complexification (viewed as real Lie group).  To prepare the proof, we start with two entirely Lie theoretic considerations.  A \emph{torus} is a connected compact abelian Lie group.

\begin{lemma}\label{lem:decomposition-tori}
  Let $T$ be a  torus and let $\tau\colon T \to T$ be an automorphism of order $2$. Every $g \in T$ can be written $g = g_0 h^2$ for certain $g_0, h \in T$ with $\tau(g_0) = g_0$ and $\tau(h)= h^{-1}$. 
\end{lemma}

\begin{proof}
  Let $\ft$ denote the Lie algebra of $T$. The automorphism of $\ft$ induced by $\tau$ will be denoted by $\tau$ as well.  We define $\ft^\tau = \{X \in \ft \mid \tau(X) = X \}$ and $\ft_{-} = \{X \in \ft \mid \tau(X) = -X\}$.  The Lie algebra decomposes as a direct sum $\ft = \ft^\tau \oplus \ft_{-}$ of subspaces.  Let $g = \exp(X) \in T$. We write $X = X_0 + Y$ with $X_0 \in \ft^{\tau}$ and $Y \in \ft_{-}$ and we define $g_0 = \exp(X_0)$ and $h = \exp(\frac{1}{2}Y)$.  Now $g = g_0 h^2$ and we observe that $\tau(g_0) = \exp(\tau(X_0)) = g_0$ and $\tau(h) = \exp(\tau(\frac{1}{2}Y)) = \exp(-\frac{1}{2}Y) = h^{-1}$. 
\end{proof}

\begin{lemma}\label{lem:ex-stable-max-torus}
  Let $K$ be a compact Lie group and let $\tau \in \Aut(K)$ be an automorphism of order two. There is a maximal torus $T \subseteq K$ which is $\tau$-stable, i.e. $\tau(T) = T$.
\end{lemma}

\begin{proof}
  Since $\tau(K^0) = K^0$ and all tori are conntained in the connected component $K^0$, we may assume that $K$ is connected. Let $\fk$ denote the Lie algebra of $K$. The correspondence between maximal tori of $K$ and maximal abelian subalgebras of $\fk$ (cf.~\cite[4.30]{Knapp}) shows that it suffices to prove the corresponding result for Lie algebras.

  We decompose $\fk$ as
  \[
  \fk = \fk^\tau \oplus \fk_{-}
  \]
  where $\fk^\tau$ is the subalgebra of $\tau$-invariant elements and $\fk_{-}$ is the $(-1)$-eigenspace of $\tau$. Let $\fa \subseteq \fk^\tau$ be a maximal abelian subalgebra of $\fk$.  Let $\fc(\fa) = \{X \in \fk \mid [X,Y] = 0 \text{ for all } Y \in \fa \: \}$ be the centralizer of $\fa$. 

  Choose a subspace $\fb \subset \fk_{-} \cap \fc(\fa)$ which is maximal abelian, i.e. is maximal with the property $[\fb,\fb] = 0$.  We define $\ft = \fa + \fb$ and we will show that $\ft$ is a maximal abelian subalgebra of $\fk$.  Clearly, $[\ft,\ft] = [\fa,\fa] + [\fa,\fb] +[\fb,\fb] = 0$ shows that $\ft$ is abelian.  Now suppose that $\fh \supseteq \fa$ is a larger abelian subalgebra.  Since $\fa \subseteq \fh$, the algebra $\fh$ lies in $\fc(\fa)$.  In particular, $\fh \cap \fk^\tau = \fa$ since $\fa$ is maximal abelian in $\fk^\tau$.  Similarly, $\fh \cap \fk_{-} = \fb$ since $\fb$ was maximal abelian in $\fc(\fa) \cap \fk_{-}$ Let $X \in \fh$ and write $X = X_{+} + X_{-}$ with $X_{+} \in \fk^\tau$ and $X_{-} \in \fk_{-}$.  For all $Z \in \fa$ one has
  \[
  0 = [X,Z] = [X_{+},Z] + [X_{-},Z] \in \fk^\tau \oplus \fk_{-}
  \]
  and therefore $X_{+} \in \fc(\fa) \cap \fk^\tau = \fa \subseteq \fh$ and so
  $X_{-} \in \fk_{-} \cap \fh = \fb$.  It follows that $\fh = \fa + \fb = \ft$.  Finally, we note that $\ft$ is $\tau$-stable, since $\fa$ and $\fb$ are contained in eigenspaces of $\tau$.
\end{proof}

\begin{proposition}\label{prop:same-rank}
  Let $G$ be a connected real Lie group with finite center and let $\sigma, \tau$ be two automorphisms of order two. Suppose that $\sigma = \inn(g) \circ \tau$ for some $g \in G$. If the fixed point groups $G^\tau$ and $G^\sigma$ are connected, then their maximal compact subgroups have the same complex rank.
\end{proposition}

\begin{proof}
  Let $Z \subseteq G$ be the center.  In the situation above, $gZ$ defines a $1$-cocycle for $\tau$, i.e. $g \tau(g) \in Z$.  Let $K \subseteq G$ be a $\tau$-stable maximal compact subgroup; the existence follows from \cite[A.4.2]{KionkePhD}.  We remark that $K^{\tau} = G^{\tau} \cap K$ is a maximal compact subgroup of $G^{\tau}$; see \cite[A.4.4]{KionkePhD}.  Since $Z$ is assumed to be finite $Z \subseteq K$ and $K/Z$ is a maximal compact subgroup of $G/Z$.  Moreover, the inclusion yields a bijection of the first non-abelian cohomology sets
  \[
  H^1(\tau,K/Z) \cong H^1(\tau,G/Z);
  \]
  see \cite[Thm.~3.1]{AnWang2008} or \cite[II.4.4]{KionkePhD}. This means, there are $k\in K$ and $h\in G$ with $k\tau(k) \in Z$ and $hk\tau(h)^{-1} \in gZ$. We define $\sigma' := \inn(k) \circ \tau$, then conjugation by $h$ defines an isomorphism between $G^{\sigma'}$ and $G^{\sigma}$. Indeed, this follows from the observation that 
  \[
  \sigma \inn(h) = \inn(g) \tau \inn(h) = \inn(g \tau(h)) \tau = \inn(hk) \tau = \inn(h) \sigma'.
  \]
  We may hence replace $\sigma$ by $\sigma'$ and assume that $g\in K$.  In this case $\tau$ and $\sigma$ stabilize $K$ and therefore also $G^\sigma \cap K = K^\sigma$ is a maximal compact subgroup of $G^{\sigma}$. Recall that $K^\tau$ and $K^\sigma$ are connected if $G^\tau$ and $G^\sigma$ are connected.
  
  From now on it suffices to consider $K$ instead of $G$.  We need to show that $K^{\tau}$ and $K^{\sigma}$ have the same rank.  The centralizer $C_K(g)$ is a closed subgroup of $K$. Since $\tau(g) \in g^{-1}Z$, the centralizer $C_K(g)$ is stable under $\tau$ (and hence also $\sigma$).  Since $g$ is contained in a maximal torus of $K$ (e.g. \cite[4.36]{Knapp}), the maximal tori in $C_K(g)$ are maximal in $K$. By Lemma \ref{lem:ex-stable-max-torus} we find a maximal torus $T \subseteq C_K(g)$ which is $\tau$-stable. In fact, $g \in T$. This follows from Theorem 4.50 in \cite{Knapp} using that $g$ centralizes $T$, the group $K$ is connected and $T$ is maximal in $K$.
  
  We use Lemma \ref{lem:decomposition-tori} to write $g = g_0 h^2$ for elements $g_0,h \in T$ which satisfy $\tau(g_0) = g_0$ and $\tau(h) = h^{-1}$. Observe that $g\tau(g) = g_0^2 \in Z$. The automorphism $\sigma' = \inn(g_0) \circ \tau$ satisfies $\sigma \circ \inn(h) = \inn(h) \circ \sigma'$ and, as above, conjugation by $h$ provides an isomorphism of the fixed point groups $K^{\sigma'}$ and $K^\sigma$. This means, we may assume that $g = g_0$ and $\tau(g) = g$. In other words, we assume that $g \in K^\tau \cap K^\sigma$. Since $K^\tau$ and $K^\sigma$ are connected, the centralizer of any element contains a maximal torus. Finally, we observe that
  \[
  C_K(g) \cap K^\sigma =  K^\tau \cap K^\sigma = K^\tau \cap C_K(g).
  \]
  This means, that a maximal torus of $K^\tau\cap K^\sigma$ is maximal in both fixed point groups.
\end{proof}


One more little observation, and afterwards the proof of Theorem~\ref{thm:f-rank-dim-mod-4} will merely be a summary of our discussion.

\begin{lemma}\label{lemma:epi-finite-index}
  Let $\varphi\colon \Gamma \to \Delta$ be a surjective homomorphism of residually finite groups. If $\varphi$ has finite kernel, then $\Gamma$ and $\Delta$ have isomorphic finite index subgroups.
\end{lemma}

\begin{proof}
  Since $\Gamma$ is residually finite, there is a finite index subgroup $\Gamma_0 \leq_{f.i.} \Gamma$ s.t. $\Gamma_0 \cap \ker(\varphi) = 1$. Then $\varphi|_{\Gamma_0}$ is an isomorphism onto a finite index subgroup of $\Delta$.
\end{proof}

\begin{proof}[Proof of Theorem~\ref{thm:f-rank-dim-mod-4}]
  We quotient out all \(\bbQ\)-simple factors of \(\alg{G}_i\) which have compact \(\bbR\)-points.  By Lemma~\ref{lemma:epi-finite-index}, the resulting groups still have profinitely isomorphic arithmetic subgroups and the values \(\dim X_i\) and \(\delta(G_i)\) are unchanged. As compact factors have finite \(\bbZ\)-points, it follows moreover that the new groups still have finite congruence kernel. Consequently, they are simply-connected~\cite[p.\,556]{PlatonovRapinchuk}.  By the Kneser--Platonov theorem~\cite[Theorem~7.12, p.\,427]{PlatonovRapinchuk}, we may therefore assume that \(\alg{G}_i\) satisfies strong approximation. From Proposition~\ref{prop:Lie-algebras-isomorphic}, we conclude that $\Lie(\alg{G}_1)(\bbQ_p) \cong \Lie(\alg{G}_2)(\bbQ_p)$ for every prime $p$.  Hence Proposition~\ref{prop:same-sign} shows that $\dim(X_1) \equiv \dim(X_2) \bmod 4$.

  Proposition~\ref{prop:inner-forms} implies that the group $\alg{G}_2 \times_\bbQ \bbR$ is an inner form of $\alg{G}_1\times_\bbQ \bbR$.  Hence an isomorphism $\varphi\colon \alg{G}_1\times_\bbQ\bbC \to \alg{G}_2\times_\bbQ \bbC$ can be chosen so that $\varphi^{-1} \tau_2 \varphi \tau^{-1}_1 = \inn(g)$ for some $g \in \alg{G}_1(\bbC)$ where $\tau_i$ denotes the involution on $\alg{G}_i(\bbC)$ induced by complex conjugation, so that $\alg{G}_i(\bbC)^{\tau_i}=\alg{G}_i(\bbR)$. Setting $\sigma = \varphi^{-1} \tau_2 \varphi$, we have $\sigma = \inn(g) \circ \tau_1$ and $\alg{G}_1(\bbC)^\sigma \cong \alg{G}_2(\bbR)$. Since the groups $\alg{G}_i$ are simply-connected, the real Lie groups $\alg{G}_i(\bbR)$ are connected; see \cite[Proposition~7.6]{PlatonovRapinchuk}. By  Proposition \ref{prop:same-rank} we have that $\rank_\bbC(\calL(K_2)\otimes \bbC) = \rank_\bbC(\calL(K_1)\otimes \bbC)$, and as $\delta(\alg{G}_i(\bbR)) = \rank_\bbC(\Lie(\alg{G}_i)(\bbC)) - \rank_\bbC(\calL(K_i)\otimes \bbC)$ we can conclude that $\delta(G_1) = \delta(G_2)$.  
\end{proof}


\section{Conclusion of main results} \label{section:conclusion}

We prove the main result, Theorem~\ref{thm:main-theorem}, and the slightly strengthened version, Theorem~\ref{thm:strengthening}. Most of what we need for the semisimple case is contained in Theorem~\ref{thm:f-rank-dim-mod-4}. In this section we put everything together and describe the reduction to the semisimple case. 

\begin{lemma}\label{lem: normal solvable subgroups}
Let $\alg G$ be a semisimple linear algebraic $\mathbb Q$-group with finite congruence kernel and $\Gamma \le \alg G(\mathbb Q)$ an arithmetic subgroup. Then $\widehat \Gamma$ has no (topologically) finitely generated infinite closed normal solvable subgroup. 
\end{lemma}

\begin{proof}
The desired property of $\widehat \Gamma$ stays unchanged by passing to finite index subgroups. The group $\alg G$ is an almost direct product 
of simple $\bbQ$-groups. A finite index subgroup of $\Gamma$ is an arithmetic subgroup of the product of simple $\bbQ$-factors whose $\bbR$-points are non-compact. Moreover, the latter product has a finite congruence kernel~\cite[p.~400]{Rapinchuk}. 
So we may and will assume that $\alg G$ contains no $\bbQ$-simple (almost) factor whose $\bbR$-points are compact. By an observation of Serre, $\alg G$ is simply connected~\cite[1.2~c)]{Serre:groupes-de-congruence}. 
 Hence $\alg G$ satisfies strong approximation~\cite[Theorem~7.12 on~p.\,427]{PlatonovRapinchuk}. Since the congruence kernel of $\alg G$ is finite, we can assume, by passing once more to a finite index subgroup, that $\widehat \Gamma$ is embedded into $\prod_p\alg G(\bbQ_p)$. By strong approximation $\widehat \Gamma$ is a compact open subgroup of $\prod_{p} U_p$, where each $U_p<\alg G(\mathbb Q_p)$ is a compact open subgroup. Let $\pr_p$ be the projection from the product to $U_p$. 
    Let $N$ be a finitely generated closed normal solvable subgroup of $\widehat \Gamma$. We have to show that $N$ is finite. 
    
    If $\pr_p(N)$ was infinite for some prime, its Lie subalgebra would be a non-trivial solvable ideal in $\calL(\alg G(\mathbb Q_p))$ contradicting semisimplicity of $\alg G$. Thus $F_p:=\pr_p(N)$ is finite for every prime $p$. 

    We know that $\Gamma$ is a subset of  $U_p \subseteq \alg G(\bbQ_p)$ and it is
    is Zariski dense in $\alg G$ by~\cite[Proposition~(3.2.11) on p.~65]{Margulis:discrete-subgroups}. 
    Since the finite, in particular algebraic, subgroup $F_p$ of $\alg G$ is normalised by the Zariski dense set $\Gamma$, we conclude that $F_p$ is a normal subgroup of $\alg G$~\cite[Proposition 1.38 on p.~31]{milne:algebraic-groups}. Thus it is contained in the center of $\alg G$ by semisimplicity. 
    So there is $e\in\bbN$ such that $F_p$ is abelian with exponent $e$ for every prime $p$. In particular, $N$ is abelian with exponent $e$. Since it is finitely generated as a profinite group, it is finite~\cite[Theorem~4.3.5 on p.~131]{ribes+zalesskii}. 
\end{proof}    
    
\begin{lemma}\label{lem: from reductive to semisimple}
Let $\alg G$ be a linear algebraic $\mathbb Q$-group and $\Gamma \le \alg G(\mathbb Q)$ an arithmetic subgroup. If $\Gamma$ has no finitely generated infinite normal solvable subgroup, then $\alg G$ is reductive and 
$\Gamma\cap\calD(\alg G)(\mathbb Q)$ has finite index in $\Gamma$ where $\calD(\alg{G})$ denotes the derived subgroup. 
\end{lemma}

\begin{proof}
 Upon passing to finite index subgroups of $\alg G$ and $\Gamma$ we may assume \(\alg{G}\) is connected so that due to \cite[Th\'eor\`eme~7.15]{Borel:introduction-aux-groupes-arithmetiques} and \cite[Section~0.24, p.\,21]{Margulis:discrete-subgroups}, we have a decomposition
\[ \alg{G} = R_u(\alg{G}) \rtimes \alg{S}\, \calD(\alg G) \]
as semidirect product of the unipotent radical \(R_u(\alg{G})\) and a reductive \(\bbQ\)-subgroup \(\alg{S}\, \calD(\alg G)\).  The latter group is an almost direct product of the central \(\bbQ\)-torus \(\alg{S}\) and the semisimple derived subgroup \(\calD(\alg G)\). 

By \cite[Corollaire~7.13.(4)]{Borel:introduction-aux-groupes-arithmetiques}, \(\Gamma\) is commensurable with the group \(\Lambda = R_u(\alg{G})(\bbZ)\, (\alg S\calD(\alg G))(\bbZ)\).  By~\cite[Corollary (3.2.9) on p.~64]{Margulis:discrete-subgroups} $\alg S(\bbZ)\calD(\alg G)(\bbZ)$ and $ (\alg S\calD(\alg G))(\bbZ)$ are commensurable. 
As arithmetic subgroups of unipotent groups are Zariski dense \cite[Lemma 3.3.3.(iii), p.\,65]{Margulis:discrete-subgroups}, it follows that \(R_u(\alg{G})(\bbZ)\) is an infinite normal nilpotent subgroup of \(\Lambda\) whenever \(R_u(\alg{G})\) is not trivial. Moreover, arithmetic groups are finitely generated. 
So our assumption on $\Gamma$ implies that \(R_u(\alg{G})\) is trivial and $\alg G$ is reductive. Since the group $\alg S(\bbZ)$ is abelian and normal in $\Lambda$ which is commensurable to $\Gamma$ it has to be finite. Thus $\Gamma\cap\calD(\alg G)(\bbQ)$ is of finite index in~$\Gamma$. 
\end{proof}

One says that a profinite group is \emph{adelic} if it isomorphic to a closed subgroup 
of some $\SL_m(\widehat \bbZ)$. 
See~\cite[p.~220]{Lubotzky-Segal:subgroup-growth} for a discussion of this notion. One easily sees that a profinite group $G$ that contains an adelic subgroup $H<SL_m(\widehat\bbZ)$ of finite index is itself adelic---via an embedding into $\SL_{m[G:H]}(\widehat \bbZ)$. 

\begin{theorem}[Platonov-Rapinchuk, Lubotzky]\label{thm: lubotzky}
Let $\alg G$ be a simply connected semisimple linear algebraic $\bbQ$-group such that each $\bbQ$-simple factor of $\alg G$ satisfies the Platonov-Margulis conjecture. Let $\Lambda<\alg G(\bbQ)$ be an 
$S$-arithmetic subgroup. If $\widehat\Lambda$ is adelic, then $\alg G$ has a finite congruence kernel. 
\end{theorem}

\begin{proof}
The group $\alg G$ is a product of its $\bbQ$-simple factors $\alg G_i$, and $\Lambda$ is 
commensurable to a product of arithmetic subgroups $\Lambda_i \subseteq \alg G_i(\bbQ)$. If each $\alg G_i$ has a finite congruence kernel then so has $\alg G$. Further, $\widehat\Lambda$ is adelic if and only each $\widehat{\Lambda_i}$ is adelic. Hence we may and will assume that $\alg G$ is $\bbQ$-simple. 

Since $\widehat \Lambda$ is an adelic group, it is boundedly generated by~\cite[Theorem~12.2 on p.~220]{Lubotzky-Segal:subgroup-growth}.  
Finally, by~\cite[Theorem 12.10 on p.~ 223]{Lubotzky-Segal:subgroup-growth}, which depends 
on the Platonov-Margulis conjecture as a global assumption, $\alg G$ has a finite congruence kernel. To be more precise, the assumption in \emph{loc.~cit.}~is that $\alg G$ is absolutely simple over a number field. But as a $\bbQ$-simple group $\alg G$ is the Weil restriction of an absolutely simple group $\alg H$ over a number field. So by \emph{loc.~cit.} $\alg H$ has 
a finite congruence kernel, hence $\alg G$ has a finite congruence kernel; see the remark at the beginning of Section~\ref{section:profinite-invariance}. 
\end{proof}

\begin{proof}[Proof of Theorems \ref{thm:main-theorem} and \ref{thm:strengthening}(i)]
By passing to finite index subgroups, we may assume that $\Gamma_1$ and $\Gamma_2$ are profinitely isomorphic.  
As in the proof before, we conclude from the congruence subgroup property of $\alg G_1$ that 
$\widehat{\Gamma}_1\cong \widehat\Gamma_2$ is adelic. 

Assume first that $\Gamma_1$ has a finitely generated infinite normal solvable subgroup. Its closure is a (topologically) finitely generated infinite closed normal solvable subgroup of $\widehat\Gamma_1\cong\widehat\Gamma_2$. The $\ell^2$-Betti numbers of $\Gamma_1$ vanish by a result of Cheeger-Gromov~\cite{cheeger+gromov}, thus $\chi(\Gamma_1)=0$. 

If $\Gamma_2$ had an infinite normal solvable subgroup, then $\chi(\Gamma_2)=0$ for the same reason and the proof would be finished. Otherwise Lemma~\ref{lem: from reductive to semisimple} would imply that $\alg G_2$ is reductive and, upon passing to finite index subgroups, $\Gamma_2$ is an arithmetic subgroup of the semisimple group $\calD(\alg G_2)$. We show that this cannot happen, thus concluding the proof in the case that $\Gamma_1$ has a finitely generated infinite normal solvable subgroup. The preimage $\Lambda$ of $\Gamma_2$ in $\widetilde{\calD(\alg G_2)}$ is commensurable with $\Gamma_2$ by~\cite[(3.2.9) Corollary on p.~64]{Margulis:discrete-subgroups}. 
    Hence $\widehat\Lambda$ is adelic because $\widehat\Gamma_1\cong\widehat\Gamma_2$ is. 
    So $\widetilde{\calD(\alg G_2)}$ has a finite congruence kernel by Theorem~\ref{thm: lubotzky} and the assumption regarding the Platonov-Margulis conjecture in Theorem~\ref{thm:strengthening}. 
    Moreover, $\widehat{\Lambda}$ contains (topologically) finitely generated infinite closed normal solvable subgroup because $\widehat\Gamma_1\cong\widehat\Gamma_2$ does. According to Lemma~\ref{lem: normal solvable subgroups} this is absurd. 
    
Next we assume that $\Gamma_1$ has no finitely generated closed normal infinite solvable subgroup. 
By Lemma~\ref{lem: from reductive to semisimple} the group $\alg G_1$ is reductive. Its derived subgroup 
$\calD(\alg G_1)$ has a finite congruence kernel as well~\cite[Lemma~2]{Raghunathan:csp-survey}. Again by Lemma~\ref{lem: from reductive to semisimple} and upon passing to a finite index subgroup of $\Gamma_1$ we may assume thus that $\alg G_1$ is semisimple and has a finite congruence kernel. By the argument at the beginning of the proof of Lemma~\ref{lem: normal solvable subgroups} we may assume that $\alg G_1$ is simply connected. Hence $\alg G_1$ has strong approximation~\cite[Theorem~7.12 on~p.\,427]{PlatonovRapinchuk}. 

By Lemma~\ref{lem: normal solvable subgroups} the group $\widehat\Gamma_1\cong \widehat\Gamma_2$ has no (topologically) finitely generated infinite closed normal solvable subgroup. In particular, $\Gamma_2$ has no finitely generated infinite normal solvable subgroup. By Lemma~\ref{lem: from reductive to semisimple} the group $\alg G_2$ is reductive, and, upon passing to finite index subgroups and replacing $\alg G_2$ by its derived subgroup, we may assume that $\Gamma_2$ is an $S$-arithmetic subgroup 
of the semisimple group $\alg G_2$. By passing to finite index subgroups once more and appealing to~\cite[(3.2.9) Corollary on p.~64]{Margulis:discrete-subgroups} and replacing $\alg G_2$ by its simply connected covering, we may assume that $\Gamma_2$ is an $S$-arithmetic subgroup of 
the simply connected semisimple group $\alg G_2$, which satisfies strong approximation by~\cite[Theorem~7.12 on~p.\,427]{PlatonovRapinchuk}.  Since $\widehat\Gamma_2$ is adelic, $\alg G_2$ has a finite congruence kernel by Theorem~\ref{thm: lubotzky}. 

 We can then apply Proposition~\ref{prop:Lie-algebras-isomorphic} to conclude that 
 $S$ contains no finite places and Theorem~\ref{thm:f-rank-dim-mod-4} to obtain that \(\dim X_1 = \dim X_2 \bmod 4\) and \(\delta(G_1) = \delta(G_2)\). Note that for the proof of Theorem~\ref{thm:main-theorem} we could just start at this point of the argument.  
  
 The semisimple Lie groups \(G_i\) possess uniform lattices \(\Lambda_i \le G_i\) and \(\Gamma_i\) is \emph{measure equivalent} to \(\Lambda_i\).  Gaboriau's proportionality principle~\cite[Th{\'e}or{\`e}me~6.3]{Gaboriau:invariants-l2} implies that \(b_n^{(2)}(\Gamma_i) = 0\) if and only if \(b^{(2)}_n(\Lambda_i) = 0\).  Borel~\cite{Borel:l2-cohomology} computed that \(b^{(2)}_n(\Lambda_i) \neq 0\) if and only if \(\delta(G_i) = 0\) and \(\dim X_i = 2n\).  As we have \(\chi(\Gamma_i) = \sum_{n \ge 0} (-1)^n b^{(2)}_n(\Gamma_i)\), it follows that
  \[
  \sign \chi(\Gamma_i) = \begin{cases}
    0 \quad &\text{ if } \delta(G_i) > 0\\
    (-1)^{\dim(X_i)/2} &\text{ if }  \delta(G_i) = 0.
  \end{cases} \]
This formula can also be deduced using Harder's Gau{\ss}-Bonnet Theorem \cite{Harder71} and
Hirzebruch's proportionality principle.
  Be aware that \(\delta(G_i) = 0\) implies that \(\dim X_i\) is even:  since every root system has an even number of roots, it follows that \(\delta(G_i)\) and \(\dim X_i\) have the same parity.  This completes the proof.
\end{proof}

The proof of the profiniteness of \(\sign \rho^{(2)}(\Gamma)\) is mostly parallel to the proof of profiniteness of \(\sign \chi(\Gamma)\).

\begin{proof}[Proof of Theorem~\ref{thm:sign-of-l2-torsion} and \ref{thm:strengthening}(ii)]
  In addition to having vanishing \(\ell^2\)-cohomology, groups of  type \((F)\) with infinite elementary amenable normal subgroups also have vanishing \(\ell^2\)-torsion~\cite[Theorem~3.113, p.\,172]{Lueck:l2-invariants}.  Hence as in the previous proof, we may assume that \(\alg{G}_1\) and \(\alg{G}_2\) are semisimple and $S$ contains no finite places.  Further, we obtain \(\dim X_1 = \dim X_2 \bmod 4\) and \(\delta(G_1) = \delta(G_2)\).  Since \(\rank_\bbQ \alg{G}_1 = \rank_\bbQ \alg{G}_2 = 0\), the arithmetic subgroups \(\Gamma_i\) are uniform lattices in \(G_i\).  Thus using the equality of topological and analytic \(\ell^2\)-torsion for closed manifolds~\cite{Burghelea-et-al:torsion}, a result of Olbrich~\cite[Theorem~1.1.(c)]{Olbrich:l2-symmetric} gives \(\rho^{(2)}(\Gamma_i) \neq 0\) if and only if \(\delta(G_i) = 1\).  From Olbrich's formulas in~\cite[Proposition~1.3]{Olbrich:l2-symmetric}, it follows moreover that if \(\delta(G_i) = 1\), then \(\sign \rho^{(2)}(\Gamma_i) = (-1)^{(\dim X_i -1)/2}\). This completes the proof of Theorem~\ref{thm:sign-of-l2-torsion}.
\end{proof}

\section{The Euler characteristic of arithmetic spin groups} \label{section:euler-computation}

In this final section we explicitly compute the Euler characteristic of arithmetic spin groups and as a particular case, we obtain the proof of Theorem~\ref{thm:euler-of-spin-groups}.

Let $V$ be a free $\bbZ$-module of finite rank $d$ with a symmetric bilinear form
$b\colon V \times V \to \bbZ$.  We will assume that the form $b$ is non-singular, i.e.\ for
every primitive vector $v \in V$, there is some $w \in V$ with $b(v,w) = 1$.

The following examples will be of interest for us. Let $m,n \geq 0$ be integers and define
$d = m+n$.  We consider $V_{m,n} = \bbZ^d$ with the standard basis $e_1,\dots,e_d$. The
bilinear form $b_{m,n}$ defined by
$$b_{m,n}(e_i,e_j) = \begin{cases}
  1 \quad & \text{ if } i = j \leq m\\
  -1 \quad & \text{ if } i = j > m\\
  0 \quad & \text{ if } i \neq j
\end{cases}
$$
is non-singular.

For every commutative ring $A$, we put $V_A = A \otimes_\bbZ V$ and we write $b_A$ for the
$A$-bilinear extension of $b$. We get an associated Clifford algebra
$$ C(V_A,b_A) = T_A(V_A) / (v^2 - b_A(v,v)) $$
as a quotient of the tensor algebra $T_A(V_A)$ of $V_A$.  As $A$-module the Clifford algebra
free and, if $e_1,\dots, e_d$ is a basis of $V$, then a basis of $C(V_A,b_A)$ is given by the
elements
$$ e(J) = e_{j_1}\cdot e_{j_2} \cdots e_{j_s} $$
for every subset $J=\{j_1,j_2,\dots,j_s\} \subseteq \{1,\dots, d\}$ with
$j_1 < j_2 < \dots < j_s$; see \cite[IV (1.5.1)]{Knus1991}. Here the convention
$e(\emptyset) = 1$ is used. As a consequence $A \otimes_\bbZ C(V,b) \cong C(V_A,b_A)$.  The
Clifford algebra is $\bbZ/2\bbZ$ graded and decomposes as $C(V_A,b_A) = C_0(V_A,b_A) \oplus C_1(V_A,b_A)$
where $C_0(V_A,b_A)$ is spanned by the $e(J)$ for sets $J$ of even cardinality.

We note that there is a unique anti-automorphism $\iota\colon C(V_A,b_A) \to C(V_A,b_A)$
with $\iota(v) = v$ for all $v \in V_A$ (of order two). Moreover, the grading yields an involution
$x \mapsto x'$ with $x=x_0+x_1$ and $x'= x_0 - x_1$ for all $x_0 \in C_0(V_A,b_A)$ and $x_1 \in C_1(V_A,b_A)$.
Composition of these two maps yields the \emph{conjugation}
$x \mapsto \overline{x} = \iota(x') = \iota(x)'$ on the Clifford algebra. 

\begin{definition}
  For a commutative ring $A$ the \emph{spin group} of $b$ over $A$ is defined by
  $$ \Spn(b)(A) = \{ g \in C_0(V_A,b_A) \mid g \overline{g} = 1 \text{ and } gV_A\overline{g} = V_A\}.$$
\end{definition}
The functor $\Spn(b)$ from the category of commutative rings to the category of groups is an
affine group scheme of finite type over $\bbZ$.  In the following
we investigate only spin groups for the forms $b = b_{m,n}$. In this case the basis vectors
satisfy $e_i\cdot e_j = - e_j\cdot e_i$  in the Clifford algebra for all $i \neq j$.
Hence, we have $\iota(e(J)) = (-1)^{|J|(|J|-1)/2}e(J)$ and therefore the identity
\begin{equation}\label{eq:Clifford-conjugation}
  \overline{e(J)} = (-1)^{|J|(|J|+1)/2} e(J)
\end{equation}
holds for every subset $J \subseteq \{1,\dots,m+n\}$.

\begin{definition}
  Let $m,n > 0$ be integers. In $\Spn(b_{m,n})(\bbZ)$ we define the principal congruence subgroup of level $4$ as
  $$ \Gamma_{m,n} = \ker\bigr( \Spn(b_{m,n})(\bbZ) \to \Spn(b_{m,n})(\bbZ/4\bbZ)\bigl).$$
\end{definition}
\begin{remark}\label{rmk:properties-arithmetic-groups}
  We decided to work with the principal congruence subgroup of level $4$ for two reasons.
  First, a classical result of Minkowski shows that the principal congruence group $\Gamma_{m,n}$ is
  torsion-free (see \cite[III.2.3]{KionkePhD} for a formulation in terms of group
  schemes). Therefore the work of Borel-Serre implies that the arithmetic group $\Gamma_{m,n}$
  is a group of type \(F\) \cite[11.1]{BorelSerre1973}.

  The second reason is that, as we shall see below, the group scheme
  $\Spn(b_{m,n})$ is not smooth at the prime $2$. Passing to the congruence subgroup of level $4$
  avoids some technicalities in the computation of the Euler characteristic.
\end{remark}

\begin{theorem}\label{thm:EulerChar-formula}
  Let $d \geq 3$ with $d = m+n$ for integers $m,n \geq 1$.
  Put $\ell = \lfloor{\frac{d}{2}}\rfloor$ and $k= \lfloor \frac{m}{2} \rfloor$.
  If $m$ and $n$ are odd, then $\chi(\Gamma_{m,n}) = 0$.
  
  If at least one of $m$ and $n$ is even, then
  $$ \chi(\Gamma_{m,n}) = (-1)^{mn/2} R(d) \binom{\ell}{k}  \prod_{j = 1}^{\ell - 1}\Bigl((2^{2j}-1)\: |\zeta(1-2j)|\Bigr)$$
  where $R(d)$ is
  \[ R(d) = \begin{cases}
    2^{5\ell^2-4\ell } (2^\ell-1) |\zeta(1-\ell)| \quad & \text{ if } d \equiv 0 \bmod 4\\
    2^{5\ell^2 - 5\ell +1} \frac{|B_{\psi,\ell}|}{\ell} \quad & \text{ if } d \equiv 2 \bmod 4\\
     2^{5\ell^2} (2^{d-1}-1) |\zeta(2-d)| \quad & \text{ if } d \equiv 1 \bmod 2.
  \end{cases}
  \]
  Here $B_{\psi,\ell}$ is the $\ell$-th generalized Bernoulli number with respect to
  the primitive Dirichlet character $\psi$ modulo $4$.
\end{theorem}
\begin{remark}
  For the definition of the generalized Bernoulli numbers we refer to \cite[p.441]{Neukirch}.
  We have $B_{\psi,\ell} = 0 $ exactly if $\ell$ is even, a case which does not occur in the
  formula. The generalized Bernoulli numbers can be computed easily. For convenience we list the first values:\\
  \begin{center}
  \begin{tabular}{l@{\hspace{1em}}*{5}{|@{\hspace{1em}}c@{\hspace{1em}}}}
    $\ell$           &  $1$ &   $3$ &   $5$ & $7$ & $9$ \\ \hline
    $B_{\psi,\ell}$ &  $-\frac{1}{2}$  &  $\frac{3}{2}$ &  $-\frac{25}{2}$ &
                                 $\frac{427}{2}$ &  $-\frac{12465}{2}$ \\
  \end{tabular}
\end{center}
\end{remark}

Assuming Theorem~\ref{thm:EulerChar-formula} for the moment, we obtain the proof of Theorem~\ref{thm:euler-of-spin-groups} as a special case.

\begin{proof}[Proof of Theorem~\ref{thm:euler-of-spin-groups}.]
  We consider the groups $\Gamma_{8,2} \subseteq \Spin(8,2)$ and $\Gamma_{4,6} \subseteq \Spin(4,6)$.
  It follows from \cite{Aka2012b} that these groups are profinitely isomorphic. We briefly recall the argument.
  The algebraic
  groups $\Spn(b_{8,2}) \times_\bbZ \bbQ$ and $\Spn(b_{4,6})\times_\bbZ \bbQ$ are simple and simply connected
  and the associated real Lie groups
  $\Spin(8,2) = \Spn(b_{8,2})(\bbR)$ and $\Spin(4,6) = \Spn(b_{4,6})(\bbR)$ are not compact.
  Hence both algebraic groups have strong approximation; see \cite[Theorem~7.2]{PlatonovRapinchuk}.
  Moreover, the quadratic forms $b_{8,2}$ and $b_{4,6}$ have Witt index $2$ and $4$ respectively, therefore
  according to \cite[11.3]{Kneser1979} the congruence kernels of $\Spn(b_{8,2})$ and $\Spn(b_{4,6})$  are trivial.
  We deduce that
  $$ \widehat{\Gamma}_{8,2} = K_{8,2} \times \prod_{p \text{ odd }} \Spn(b_{8,2})(\bbZ_p)$$
  and
  $$ \widehat{\Gamma}_{4,6} = K_{4,6} \times \prod_{p \text{ odd }} \Spn(b_{4,6})(\bbZ_p)$$
  where $K_{m,n} = \ker\bigl( \Spn(b_{m,n})(\bbZ_2) \to \Spn(b_{m,n})(\bbZ/4\bbZ)\bigr)$ is
  the open compact principal congruence subgroup of level $4$.  However, the forms $b_{8,2}$
  and $b_{4,6}$ are isometric over $\bbZ_p$ for every prime number $p$; see \cite[Cor.3]{Aka2012b}.
  Thus the group schemes $\Spn(b_{8,2}) \times_\bbZ \bbZ_p$ and
  $\Spn(b_{4,6})\times_\bbZ \bbZ_p$ are isomorphic for every prime $p$.  In particular,
  $ \Spn(b_{8,2})(\bbZ_p) \cong \Spn(b_{4,6})(\bbZ_p)$ and $K_{8,2} \cong K_{4,6}$; we deduce
  that the profinite completions are isomorphic.

  Now we use Theorem \ref{thm:EulerChar-formula} to compute the Euler characteristic.
  We have $d = 10$ and $\ell = 5$ and thus we obtain $R(d) = 2^{100} \cdot  5$.
  Since $\zeta(-1) = -\frac{1}{12}$, $\zeta(-3) = \frac{1}{120}$, $\zeta(-5) = -\frac{1}{252}$
  and $\zeta(-7) = \frac{1}{240}$ (c.f.\ \cite[\S 1.5]{Edwards1974})
  the product evaluates as
  $$\prod_{j=1}^4\Bigr( (2^{2j}-1) \: |\zeta(1-2j)| \Bigl) = \frac{3 \cdot 15 \cdot 63 \cdot 255}{12 \cdot 120 \cdot 252 \cdot 240} = \frac{17}{ 2^{11} }.$$
  For $m = 8$ we have $k=4$ and with $\binom{5}{4} = 5$ we obtain
  $$\chi(\Gamma_{8,2}) = 2^{89} \cdot 5^2 \cdot 17.$$
  For $m=4$ we have $k=2$ and since $\binom{5}{2} = 10$ we have
  \[ \chi(\Gamma_{4,6}) = 2^{90} \cdot 5^2 \cdot 17. \qedhere \]
\end{proof}

In order to prove Theorem \ref{thm:EulerChar-formula} we need some preparation.
For simplicity we will write from now on $b$ for $b_{m,n}$ and we set $\alg{G} = \Spn(b_{m,n})$.
As a first step, we determine the Lie algebra of $\alg{G}$.
It follows from the next lemma that the group scheme $\alg{G} \times_\bbZ \bbZ_p$ is smooth if $p \neq 2$, but
is not smooth for $p=2$. This problem forces us to be more careful when dealing with the prime $p=2$.

\begin{lemma}\label{lem:LieAlgebraSpin}
  Let $A$ be a commutative ring and let $\ann(2A) = \{a \in A \mid 2a = 0\}$
  be the annihilator of $2A$.
  The Lie algebra of $\alg{G}$ over $A$ is isomorphic to the Lie subalgebra of $C_0(V_A,b_A)$
  given by
  $$\Lie(\alg{G})(A) \cong \bigoplus_{|J|=2}A e(J)
  \oplus \bigoplus_{\substack{J \text{even}\\ |J| \neq 2 }}\ann(2A) e(J)$$
  where the sums run over subsets $J \subseteq \{1,\dots,d\}$ of even cardinality.
\end{lemma}
\begin{proof}
  Consider the ring $A[\varepsilon]$ with $\varepsilon^2=0$.
  Recall that the  Lie algebra is defined as
  $$\Lie(\alg{G})(A) = \{X \in C_0(V_A,b_A) \mid  1 +\varepsilon X \in \alg{G}(A[\varepsilon]) \}.$$
  Let $X \in C_0(V_A,b_A)$.
  Write
  $$ X = \sum_{J \text{ even }} x_J e(J)$$
  with coefficients $x_J \in A$ and define $g = 1+\varepsilon X$. Note that $g^{-1} = 1- \varepsilon X$.
  We determine under which conditions $X \in \Lie(\alg{G})(A)$.
  
  Using \eqref{eq:Clifford-conjugation} we see that
  $1 = g \overline{g} = 1 + \varepsilon (X + \overline{X}) $
  holds exactly if $x_J \in \ann(2A)$ for all $J$ with $|J| \equiv 0 \bmod 4$.  Moreover, $g$
  satisfies $g e_i g^{-1} \in V_{A[\varepsilon]}$ if and only if $X e_i - e_i X \in V_{A}$ for
  all $i$.  Let $J \subseteq \{1,\dots, d\}$ be a set with an even number of elements.  Then
  $e(J)e_i -e_ie(J) = 0$ if $i \not\in J$.  However, if $i\in J$, then
  $$e(J)e_i - e_i e(J) = e_i^2 (-1)^{|\{j\in J\mid j > i\}|} 2 e(J\setminus\{i\}).$$  We deduce
  that $Xe_i - e_iX \in V_A$ is satisfied precisely when $x_J \in \ann(2A)$ for all $J$ with
  $|J| > 2$.  We leave it to the reader to verify that the Lie algebra structure is indeed
  induced by the commutator bracket on $C_0(V_A,b_A)$. For instance,
  one can use the formula given in \cite[II,\S4,4.2]{DemazureGabriel}.
\end{proof}

\begin{proof}[Proof of Theorem \ref{thm:EulerChar-formula}]
  We first fix some notation.  Let $\alg{G}(\bbR) = \Spin(m,n)$ be the associated real spin
  group.  The Lie algebra $\calL(\Spin(m,n))$ will be denoted by $\fg$. We identify $\fg$ with
  a Lie subalgebra of the Clifford algebra $C(V_\bbR,b_\bbR)$; c.f.~Lemma
  \ref{lem:LieAlgebraSpin}.  The vectors $e(J)$ where $J$ runs through the two-element subsets
  of $\{1,\dots,d\}$ are a basis of $\fg$.  The subalgebra $\fk$ spanned by the $e(J)$ where
  $J \subseteq \{1,\dots,m\}$ or $J \subseteq \{m+1,\dots,m+n\}$ is maximal compact.  The
  corresponding maximal compact subgroup will be denoted by $K_\infty$.  A Cartan
  decomposition is given by $ \fg = \fk \oplus \fp $ where $\fp$ is spanned by the $e(J)$ with
  $J =\{i,j\}$ satisfying $i \leq m < j$.

  Let $X = \Spin(m,n)/K_\infty$ denote the
  associated Riemannian symmetric space.  Note that $\dim(X) = mn$ and
  $\dim(\alg{G}\times_\bbZ \bbQ) = d(d-1)/2$ where $d = m+n$.  Since $X$ is simply connected
  and $\Gamma_{m,n}$ acts freely and properly on $X$, the quotient space $X/\Gamma_{m,n}$ is
  the classifying space of $\Gamma_{m,n}$. We will calculate the Euler characteristic of this
  space.  If $m$ and $n$ are odd, then the Euler-Poincar\'e measure on $\Spin(m,n)$ vanishes
  and $\chi(\Gamma_{m,n}) = 0$; see\ \cite{Serre1971} or \cite[Thm.~3.1]{Kionke2014b}. From
  now on we assume that $\dim(X) = mn$ is even.

  The linear algebraic group $\alg{G} \times_\bbZ \bbQ$ is simple and simply
  connected~\cite[Theorem~24.61 on p.~534]{milne:algebraic-groups}.  The
  associated real Lie group $\alg{G}(\bbR) = \Spin(m,n)$ is not compact, since we assume
  $m,n \geq 1$. We infer that $\alg{G}$ has strong approximation; see
  \cite[Thm.~7.12]{PlatonovRapinchuk}.  It follows that the inclusion
  $\Spin(m,n) \to \alg{G}(\bbA)$ induces a homeomorphism
  $$ \alg{G}(\bbQ) \backslash \alg{G}(\bbA)/K_\infty K_f \cong X/\Gamma_{m,n} $$
  where $K_f = K^{(2)}_{m,n} \times \prod_{p \text{ odd }} \alg{G}(\bbZ_p)$ is an
  open compact subgroup with
  $$ K^{(2)}_{m,n} = \ker\bigl( \alg{G}(\bbZ_2) \to \alg{G}(\bbZ/4\bbZ)\bigr).$$
  
  We will compute the Euler characteristic using the adelic formula given in Theorem 3.3 in
  \cite{Kionke2014b}. We choose $B$ to be the symmetric bilinear form on $\Lie(\alg{G})(\bbQ)$
  for which the vectors $e(J)$ are an
  orthonormal basis.
  The form $B$ is \emph{nice} in the sense of \cite{Kionke2014b}, i.e., the Cartan decomposition
  on $\fg = \fk \oplus \fp$ given above is orthogonal. Moreover, $B$ induces a volume form $\vol_B$ on $\alg{G}(\bbQ_v)$ at every place of
  $v$ of $\bbQ$. Now Theorem 3.3 in \cite{Kionke2014b} yields
  \begin{equation}\label{eq:AdelicFormula}
    \chi( X/\Gamma_{m,n} )
    = (-1)^{mn/2} \frac{|W(\fg_\bbC)| \: \tau(\alg{G})}{|W(\fk_\bbC)|}
    \vol_B(G_u)^{-1} \vol_B(K_f)^{-1}
  \end{equation}
  where $\tau(\alg{G})$ is the Tamagawa number of $\alg{G}$ and $W(\fg_\bbC)$ and
  $W(\fk_\bbC)$ denote the Weyl groups of the complexified Lie algebras of $\fg_\bbC$ and
  $\fk_\bbC$ respectively. Moreover, $G_u$ denotes the compact dual group,i.e.,\ the compact group $\Spin(d)$
  in our case.

  Now we evaluate the terms in the formula step by step. We put $\ell = \lfloor\frac{d}{2}\rfloor$ and
  $k = \lfloor\frac{m}{2}\rfloor$ and $k' = \lfloor\frac{n}{2}\rfloor$.
  Observe that $\ell = k+k'$ since we excluded the case that both $m$ and $n$ are odd.

  In evaluating the volume $\vol_B(K_f)$ there is, however, a subtle point: the adelic formula
  in \cite{Kionke2014b} is based on the assumption that the underlying group scheme is smooth
  over $\bbZ$. This assumption is only used in evaluating $\vol_B(K_f)$.  As we have seen in
  \ref{lem:LieAlgebraSpin} our group scheme $\alg{G}$ is not smooth over $\bbZ$ since there is
  a problem at the prime $2$. In the last step we shall take care of this problem.

  \medskip

  \emph{Tamagawa number:}  $\tau(\alg{G}) = 1$.\\
  Since $d \geq 2$ the spin group $\alg{G}\times_\bbZ \bbQ$ is semi-simple and simply
  connected. The assertion follows from Kottwitz' Tamagawa number theorem \cite{Kottwitz1988}.
  For spin groups this was already observed by Tamagawa and Weil.

  \medskip

  \emph{Orders of Weyl groups:}  $ \frac{|W(\fg_\bbC)|}{|W(\fk_\bbC)|} = 2 \binom{\ell}{k}$.\\
  If $d = 2\ell$ is even, then $\fg_\bbC$ is a simple Lie algebra of type $D_\ell$
  and $\fk_\bbC$ is a product of simple Lie algebras of type $D_k$ and $D_{k'}$.
  The table in \cite[p.~66]{Humphreys}, yields
  $$ \frac{|W(\fg_\bbC)|}{|W(\fk_\bbC)|} = \frac{2^{\ell-1}\ell!}{2^{k-1}k! \: 2^{k'-1}k'!} = 2 \binom{\ell}{k}.$$
  Similarly, if $d = 2\ell +1$, then $\fg_\bbC$ is a simple Lie algebra of type $B_\ell$ and the Weyl group has order
  $2^\ell \ell!$; see \cite[p.~66]{Humphreys}. Now
  $\fk_\bbC$ is a product of two simple Lie algebras either of types $B_k$ and $D_{k'}$ or of
  types $D_{k}$ and $B_{k'}$. A short calculation yields the formula.

  \medskip

  \emph{Volume of $G_u$:} $\vol_B(G_u) = 2^{(3d-d^2)/2}\prod_{j=2}^d \pi^{j/2} \Gamma(j/2)^{-1}$.\\
  The compact dual group $G_u$ is $\Spin(d)$. Since $\Spin(d)$ is a two-fold covering of
  $\SO(d)$ we obtain $\vol_B(\Spin(d)) = 2 \vol_B(\SO(d))$.  However, we have to relate the
  induced left invariant Riemann metric $B$ to the standard left invariant metric $\gamma$ on
  $\SO(d)$. More precisely, the vectors $v_{i,j} = E_{i,j} - E_{j,i}$ with $i<j \leq d$, where
  $E_{i,j}$ denotes the elemenary matrix with entry $1$ in position $(i,j)$, form a basis of
  the Lie algebra $\fso(d)$. At the identity $\gamma$ is the symmetric bilinear form for which
  $(v_{i,j})_{i<j}$ is an orthonormal basis.  Using induction one shows that the volume with
  respect to $\gamma$ is
  $$\vol_\gamma(\SO(d)) = \prod_{j=2}^d\vol(S^{j-1}) = \prod_{j=2}^d 2 \frac{\pi^{j/2}}{\Gamma(j/2)}.$$
  The tangent map of the projection $p\colon \Spin(d) \to \SO(d)$ maps the
  basis vector $e(I)$ with $I = \{i,j\}$ and $i < j$ to $2 v_{i,j}$,
  therefore $B = \frac{1}{4} \gamma$ and
  $$ \vol_B(\SO(d)) = 2^{-d(d-1)/2}\vol_\gamma(\SO(d)).$$

  \medskip

  \emph{Local volume $\vol_B(K_f)$:}\\
  Here we obtain a formula for the local volume
  $$\vol_B(K_f) = \vol_B(K_{m,n}^{(2)})\prod_{p \text{ odd }} \vol_B(\alg{G}(\bbZ_p)).$$
  As in
  \cite{Kionke2014b} we use the smoothness of $\alg{G}\times_\bbZ \bbZ_p$
  for all \emph{odd} primes to apply  Weil's formula
  $\vol_B(\alg{G}(\bbZ_p)) = |\alg{G}(\bbF_p)|p^{-\dim \alg{G}}$.
  The special $p=2$ is discussed in Lemma
  \ref{lem:non-smooth-formula} below, which yields
  $\vol_B(K^{(2)}_{m,n}) = 4^{-d(d-1)/2} = 2^{-d(d-1)}$.
  It remains to evaluate the infinite product over all odd primes
  $$\prod_{p \text{ odd }} |\alg{G}(\bbF_p)|^{-1} p^{d(d-1)/2}.$$
  Recall that over $\bbF_p$ there are exactly two quadratic forms in $d$ variables.
  They are uniquely determined by their discriminant in $\bbF_p^\times/(\bbF_p^\times)^2$; see
  \cite[Thm.~3.8]{Scharlau1985}.
  We note that the canonical map $\alg{G}(\bbF_p) \to \alg{\SO}(m,n)(\bbF_p)$ has a two-element kernel
  and, as $d\geq 3$, the image has index $2$ in $\alg{\SO}(m,n)(\bbF_p)$, thus
  $|\alg{G}(\bbF_p)| = |\alg{\SO}(m,n)(\bbF_p)|$.

  \smallskip

  Case 1: $d=2\ell$ is even.\\
  Let $p$ be an odd prime number.
  By assumption $m$ and $n$ are even, hence the discriminant $\det(b_{m,n}) = 1$.
  If a quadratic space $(V,q)$ of dimension $d$ over $\bbF_p$ splits
  as an orthogonal sum of hyperbolic planes, then we say that $q$ is of $\oplus$-type.
  Otherwise, $q$ has an anisotropic kernel of dimension $2$ and we say that $q$ is of $\ominus$-type.

  \smallskip
  
  Case 1a: $d \equiv 0 \bmod 4$.\\
  In this case $\ell$ is even, thus
  $\det(b_{m,n}) = 1 = (-1)^\ell = \det(H^{\perp \ell})$
  where $H$ denotes the hyperbolic plane. We deduce that $b_{m,n}$ is of $\oplus$-type over $\bbF_p$.
  In this case 
  $|\alg{G}(\bbF_p)| = p^{\ell(\ell-1)} (p^\ell -1) \prod_{j=1}^{\ell-1}(p^{2j}-1)$; see \cite[p.~147]{ArtinGA}.
  We obtain
  \begin{align*}
    \prod_{p \text{ odd }} |\alg{G}(\bbF_p)|^{-1} p^{d(d-1)/2}
    & = \prod_{p \text{ odd }} (1-p^{-\ell})^{-1}\prod_{j=1}^{\ell-1}(1-p^{-2j})^{-1}\\
    &= \zeta(\ell)(1-2^{-\ell}) \prod_{j=1}^{\ell-1} \zeta(2j) (1-2^{-2j}),
  \end{align*}
  where $\zeta$ is the Riemann zeta function.

  \smallskip

  Case 1b: $d \equiv 2 \bmod 4$.\\
  In this case $\ell$ is odd and as in Case 1a) we see that $b_{m,n}$ is of $\oplus$-type exactly
  if $-1$ is a square, i.e.  $p \equiv 1 \bmod 4$.
  Let $\psi$ denote the unique primitive Dirichlet character modulo $4$, then the order of the spin group is
  $|\alg{G}(\bbF_p)| = p^{\ell(\ell-1)} (p^\ell - \psi(p)) \prod_{j=1}^{\ell-1}(p^{2j}-1)$; see \cite[p.~147]{ArtinGA}.
  Using this we obtain 
   \begin{align*}
    \prod_{p \text{ odd }} |\alg{G}(\bbF_p)|^{-1} p^{d(d-1)/2}
    & = \prod_{p \text{ odd }} (1-\psi(p)p^{-\ell})^{-1}\prod_{j=1}^{\ell-1}(1-p^{-2j})^{-1}\\
    &= L(\psi,\ell) \prod_{j=1}^{\ell-1} \zeta(2j) (1-2^{-2j}),
  \end{align*}
  where $L(\psi,s)$ is the Dirichlet $L$-function attached to $\psi$.
  
  \smallskip

  Case 2: $d=2\ell + 1$ is odd.\\
  In this case the order is $|\alg{G}(\bbF_p)| = p^{\ell^2} \prod^{\ell}_{j=1}(p^{2j}-1)$; see \cite[p.~147]{ArtinGA}.
  Consequently, we obtain the formula
  \begin{equation*}
    \prod_{p \text{ odd }} |\alg{G}(\bbF_p)|^{-1} p^{d(d-1)/2}
    =  \prod_{j=1}^{\ell} \zeta(2j) (1-2^{-2j}),
  \end{equation*}

  \medskip

  Finally, we multiply the terms and simplify using the functional equations of the Riemann
  zeta function and the Dirichlet $L$-function.  More precisely, the functional equation of
  the $\zeta$-function \cite[VII.~(1.6)]{Neukirch} and the well-known identity
  $\Gamma(\frac{1}{2}-j)\Gamma(\frac{1}{2}+j) = (-1)^j \pi$
  imply that
  $$ \zeta(2j) \:  \pi^{-2j/2} \Gamma(\frac{2j}{2})\:  \pi^{-(2j+1)/2}\Gamma(\frac{2j+1}{2}) = |\zeta(1-2j)|.$$
  This identity makes it possible to combine one factor of the product in $\vol_B(K_f)^{-1}$
  with two consecutive factors of the product occuring in $\vol_B(G_u)^{-1}$.
  If $d$ is even, there is a remaining term which needs to be simplified.
  If $d \equiv 0 \bmod 4$, then $\zeta(\ell)\pi^{-\ell}\Gamma(\ell) = 2^{\ell-1} |\zeta(1-\ell)|$
  as can be seen using the functional equation.
  For the case $d \equiv 2 \bmod 4$ the functional equation of the $L$-function \cite[VII.~(2.8)]{Neukirch}
  yields $$L(\psi,\ell)\pi^{-\ell}\Gamma(\ell) =  2^{-\ell} |L(\psi,1-\ell)|.$$
  
  Eventually, we use $L(\psi,1-\ell) = -\frac{B_{\psi,\ell}}{\ell}$ to express the special $L$-value
  in terms of generalized Bernoulli numbers; see \cite[VII.~(2.9)]{Neukirch}.
\end{proof}

\begin{lemma}\label{lem:non-smooth-formula}
  In the notation above, we have
  $\vol_B(K^{(2)}_{m,n}) = 4^{-d(d-1)/2}$.
\end{lemma}
\begin{proof}
  We construct an explicit chart which will allow us to compute the volume.  The exponential
  series converges on $4 C_0(V_{\bbZ_2},b)$ and defines an analytic function
  with values in $1 + 4 C_0(V_{\bbZ_2},b)$.  Let $x \in 4
  \Lie(\alg{G})(\bbZ_2)$. Then $\overline{x} = -x$ commutes with $x$, thus
  $\exp(x)\overline{\exp(x)} = \exp(x-x) = 1$. Moreover, we claim that
 $\exp(x)v\overline{\exp(x)} \in V_{\bbZ_2}$ for every $v \in V_{\bbZ_2}$.  Indeed, consider
  the endomorphism $\mathrm{ad}_x$ of $V_{\bbZ_2}$ defined by $\mathrm{ad}_x(v) = xv - vx$. Then a short calculation
  yields
  $$\exp(x)\: v \: \overline{\exp(x)} = \exp(\mathrm{ad}_x) (v) \in V_{\bbZ_2}.$$
  We deduce that the exponential function maps $4 \Lie(\alg{G})(\bbZ_2)$ to the group $K^{(2)}_{m,n}$.
  Similarly, the logarithmic series $L(1+a) = \sum_{k=1}^\infty \frac{(-1)^{k-1}}{k}a^k$
  converges on $K_{m,n}^{(2)}$ and with similar arguments one verifies that
  $L(K_{m,n}^{(2)})\subseteq 4\Lie(\alg{G})(\bbZ_2)$.
  Since $\exp$ and $L$ are inverses of each other, we deduce that
  $$\exp \colon  4\Lie(\alg{G})(\bbZ_2) \to K^{(2)}_{m,n}$$
  is an analytic isomorphism.
  It is straightforward to check that the pullback of the volume density on $K^{(2)}_{m,n}$
  via the exponential map to $4\Lie(\alg{G})(\bbZ_2)$ yields the standard volume.
  As a consequence we obtain $\vol_B(K^{(2)}_{m,n}) = 4^{-\dim \alg{G}}$.
\end{proof}

\begin{example} \label{dim_infty_not_profin}
  As discussed in the introduction our methods do not suffice to prove Theorem
  \ref{thm:main-theorem} for $S$-arithmetic groups. The sign of the
  Euler characteristic of an arithmetic group depends only on the archimedean place,
  whereas for $S$-arithmetic
  groups the sign depends on all places in $S$. This makes it necessary to understand the subtle
  interplay between the places. One class of examples which illustrates
  this behavior quite well are special linear groups over quaternion algebras.
  Other intruiging examples arise from $S$-arithmetic spin groups, as we will see now.

  We consider the set  $S=\{2, \infty\}$ of places of the field $\bbQ$ and the two groups
  $\Spn(b_{4,1})$ and $\Spn(b_{2,3})$.
  The quadratic forms $b_{4,1}$ and $b_{2,3}$ are equivalent over $\bbZ_p$ for every prime $p > 2$.
  Indeed, for odd primes $b_{3,0}$ is isotropic over $\bbZ_p$; see \cite[Lemma 1.7, p.41]{Cassels1978}.
  Hence $b_{3,0}$ splits into a hyperbolic plane and a $b_{0,1}$.  This proves the assertion, since
  $b_{1,1}$ is equivalent to the hyperbolic plane over $\bbZ_p$.
  We note further that $b_{3,0}$ is anisotropic over $\bbQ_2$ (\cite[Lemma 2.5, p.59]{Cassels1978}) and we deduce that
  the Witt index of $b_{4,1}$ is $1$ over $\bbQ_2$.

  The rank of a spin group over a field of characteristic $\neq 2$ is the Witt index of the defining quadratic form.
  We deduce that
  \begin{align*}
    \rank_S\Spn(b_{4,1}) &= \rank_{\bbR} \Spn(b_{4,1}) + \rank_{\bbQ_2} \Spn(b_{4,1}) = 1+ 1 = 2,\\
    \rank_S\Spn(b_{2,3}) &=  \rank_{\bbR} \Spn(b_{2,3}) + \rank_{\bbQ_2} \Spn(b_{2,3}) = 2+ 2 = 4. 
  \end{align*} 
  In particular, the $S$-arithmetic groups $\Delta_1 = \Spn(b_{4,1})(\bbZ[1/2])$ and $ \Delta_2 = \Spn(b_{2,3})(\bbZ[1/2])$ have
  the congruence subgroup property (cf.~\cite[Theorem~5]{Kammeyer-Sauer:spinor-groups})
  and are hence profinitely isomorphic.

  The symmetric space $X_1$ associated to $\Spin(4,1)$ has dimension $4$, whereas
  the symmetric space $X_2$ of $\Spin(2,3)$ has dimension $6$.
  In particular, $\dim X_1 \not\equiv \dim X_2 \bmod 4$ and the useful
  Theorem~\ref{section:profinite-invariance} fails in the $S$-arithmetic case.
  However, the Euler characteristics of $\Delta_1$ and $\Delta_2$ have nevertheless the same sign.
  Using Serre's description of the Euler-Poincar\'e measure \cite{Serre1971}
  we see that
  the sign of the Euler characteristic of $\Delta_1$ is
  $$\sign(\chi(\Delta_1)) = (-1)^{\dim(X_1)/2} \cdot (-1)^{\rank_{\bbQ}(\Spn(b_{4,1}))} = -1$$
  and the sign of the Euler characteristic of $\Delta_2$ is
  $$\sign(\chi(\Delta_2)) = (-1)^{\dim(X_2)/2} \cdot (-1)^{\rank_{\bbQ}(\Spn(b_{2,3}))} = -1.$$
  The problem $\dim X_1 \not\equiv \dim X_2 \bmod 4$ is repaired by the change of the $\bbQ_2$-rank
  modulo $2$, i.e.\
  $\rank_{\bbQ_2} \Spn(b_{4,1}) \not\equiv \rank_{\bbQ_2} \Spn(b_{2,3}) \bmod 2$.
  \end{example}

 \bibliography{literatur}
 \bibliographystyle{amsplain}
\end{document}